%% file: ms.tex
\documentclass[12pt]{amsart} 
\usepackage{amsfonts}
\usepackage{amssymb}
\usepackage{amsmath}
\usepackage{amsthm}
\usepackage{enumerate}
\usepackage{graphicx}
\usepackage[backend=biber, style=alphabetic]{biblatex}
\usepackage[super]{nth}

\usepackage[pdftex,pdfpagelabels,bookmarks,hyperindex,hyperfigures,hidelinks,
            pdfauthor={Thomas Gunnison Brooks},
            pdftitle={3-Manifolds with Constant Ricci Eigenvalues (lambda, lambda, 0)},
            pdfsubject={Riemannian geometry}]{hyperref}

\usepackage[hmargin=1 in, vmargin= 1.25 in]{geometry}

\author{Thomas G. Brooks}
\date{\today}

\title{$3$-Manifolds with Constant Ricci Eigenvalues $(\lambda, \lambda, 0)$}

\input{base_commands.tex}
\renewcommand\vec{\mathbf}

\begin{document}

\begin{abstract}
    We consider complete Riemannian $3$-manifolds whose Ricci tensors have constant eigenvalues $(\lambda, \lambda, 0)$.
    When $\pi_1$ is finitely generated, we classify the topology of such manifolds by showing that they have a free fundamental group if non-trivial and that every free group is obtained.
    We give a description up to isometry, when the metric is locally irreducible or when it is analytic.
\end{abstract}

\maketitle

A manifold $M$ is \emph{curvature homogeneous} if for any two points $p,q \in M$ there is an isometry $f: T_p M \rightarrow T_q M$ of their tangent spaces such that the curvature tensor $R_p$ is the pullback $f^*(R_q)$ of $R_q$.
The study of curvature homogeneous manifolds originates with \cite{singer} where the following question was posed.
\begin{question*}[Singer]
    Are all curvature homogeneous manifolds locally homogeneous?
\end{question*}
This was answered in the negative in \cite{sekigawa}, see \eqref{eqn:sekigawa_form}.
The following open conjecture provides a related motivation \cite{CN2_book}.

\begin{conjecture*}[Gromov]
Fix a compact manifold $M$ and a curvature tensor $R$.
Then the space of curvature homogeneous metrics on $M$ (up to isometry) which have curvature tensor $R$ is finite dimensional.
\end{conjecture*}
The Sekigawa examples give two infinite dimensional moduli spaces of complete curvature homogeneous metrics, but these are not compact.

Since, in dimension 3, the Ricci tensor determines the curvature tensor, curvature homogeneous manifolds are those manifolds whose Ricci eigenvalues are constant. 
Such manifolds are well studied locally \cite{distinct_ricci_eigs, CN2_book, constant_principal_curvatures}.
In this paper, we study the global behavior of the special case where the Ricci eigenvalues are $(\lambda, \lambda, 0)$.
This is the only case where $M$ can be isometric to a product metric locally.
Since some regions may be locally irreducible while others are locally a product metric, this condition allows for a large class of examples differing in geometry as well as topology.

By scaling, we will assume from now on that $\lambda = -1, 0,$ or $1$.
If $\lambda = 0, 1$ and the metric is complete, then it is well known that the universal covers of such a manifold splits isometrically as $\mathbb{E}^2 \times \R$ or $\mathbb{S}^2 \times \R$, see Lemma~\ref{lemma:c_stays_nilpotent}.
Hence we will assume that $\lambda = -1$.

Apart from quotients of $\mathbb{H}^2 \times \R$, previous examples of complete manifolds with Ricci eigenvalues $(-1,-1,0)$ were simply connected \cite{KTV_curv_homog, szabo2}.
We classify the fundamental group, under the assumption that it is finitely generated, which classifies their topology since they are Hadamard manifolds .

\begin{theorem}
    \label{thm:pi1}
    Suppose that $M$ is complete and has Ricci eigenvalues $(-1,-1,0)$ and $\pi_1(M)$ is finitely generated.
    If its universal cover $\widetilde M$ is irreducible, then $\pi_1(M)$ is a free group.
    If $M$ is locally irreducible everywhere, then $\pi_1(M)$ is either trivial or $\Z$.
    Moreover, any countably generated free group is obtained as $\pi_1(M)$ of some such metric.
\end{theorem}

Consider now the locally irreducible manifolds with Ricci eigenvalues $(-1, -1, 0)$.
\begin{theorem}
\label{thm:continuous_coords}
    Suppose that $M$ is a complete, simply connected manifold with Ricci eigenvalues $(-1, -1, 0)$ such that $M$ is locally irreducible everywhere.
    Then there exist  coordinates $(x,u,v)$ on $M$, which are locally Lipschitz and smooth in an open, dense subset on which the metric has the form, 
    \begin{equation}
        \label{eqn:g_curv_homog}
        g = (\cosh u - h(x) \sinh u)^2 dx^2 + (du - f(x) v \; dx)^2 + (dv + f(x) u \; dx)^2,
        \tag{$\star$}
    \end{equation}
    for some functions $f: \R \rightarrow \R$, $h: \R \rightarrow [-1,1]$ where $f$ is $C^{1,1}$ and $h$ is the derivative of a Lipschitz function.
    Moreover, $f,h$ are determined up to sign by a choice of a base-point in $M$.
\end{theorem}

The following theorem gives a partial converse to this.
\begin{theorem}
    \label{thm:irreducible_examples}
    Let $f: \R \rightarrow \R$ be smooth and $H: \R \rightarrow \R$ be Lipschitz with Lipschitz constant 1 and smooth on an open dense set $S \subset \R$.
    Let $h = H'$ and assume that
    \begin{equation}
    \label{eqn:f_h_property}
    f^{(k)}(x) h^{(\ell_1)}(x) \cdots h^{(\ell_m)}(x) \rightarrow 0
    \end{equation}
    on $\R \setminus S$, for any $m, k, \ell_i \geq 0$.

    Then a metric $g$ of the form~\eqref{eqn:g_curv_homog} in Lipschitz coordinates $(x, u, v)$ is a smooth, complete metric and has Ricci eigenvalues $(-1, 1, 0)$.
\end{theorem}
Notice that we do not require that $h$ to even be continuous or well-defined outside of $S$.
We give examples where $h$ is non-continuous and the coordinates $(x, u, v)$ are non-smooth on a Cantor set, although the metric $g$ is $C^\infty$.

As a consequence of these theorems we obtain a complete description of the case where $M$ is analytic.
\begin{corollary}
\label{cor:analytic}
Suppose $M$ is complete, analytic, and irreducible, and its Ricci eigenvalues are $(-1, -1, 0)$.
Then there exist analytic functions $f: \R \rightarrow \R$, $h: \R \rightarrow [-1, 1]$ so that $M$ is isometric to a metric as in \eqref{eqn:g_curv_homog}.
Conversely, for any such $f, h$,~\eqref{eqn:g_curv_homog} has Ricci eigenvalues $(-1, -1, 0)$, is complete and analytic, and is irreducible unless $f$ is identically zero.
\end{corollary}

We now outline the proofs of these results.
Define the splitting tensor $C(X) := -\nabla_X T$ where $T$ is the unit eigenvector of the Ricci tensor with eigenvalue $0$.
By a de Rham splitting, $C$ is $0$ in a neighborhood if and only if $M$ locally splits with an $\R$ factor.
We define $M_{split}$ to be the points in $M$ which have a neighborhood locally isometric to a product metric $\Sigma^2 \times \R$ and $M_{irred}$ to be the complement of $M_{split}$.
The main geometric structure of these manifolds, when simply connected, is a Lipschitz foliation $\mathcal{F}$ on $M_{irred}$ by complete, totally geodesic, flat planes.

To prove Theorem~\ref{thm:pi1}, we show that no leaf in $\mathcal{F}$ can be invariant under a non-trivial isometry.
From this, we construct a Lyndon length function \cite{lyndon} on $\pi_1(M)$ by counting the leaves of a discrete subset of $\mathcal{F}$ that lie between a point in $\widetilde M$ and its image under an isometry.

We prove Theorem~\ref{thm:continuous_coords} by showing that in any connected component of $M_{irred}$, there exists a $C^{1,1}$ path orthogonal to $\mathcal{F}$.
From these, we construct the coordinates \eqref{eqn:g_curv_homog}, where $h$ is the geodesic curvature of this curve and $f$ is the norm of the $C$ along the curve.

In order to prove Theorem~\ref{thm:irreducible_examples}, we start with a foliation of $\mathbb{H}^2$ by geodesics and let $\gamma$ be the $C^{1,1}$ curve orthogonal to the geodesics.
In the coordinates on $\mathbb{H}^2 \times \R$ given by $(x,u,v) \mapsto (\exp_{\gamma(x)}(u U(x)), v)$, where $U$ is a unit vector field orthogonal to $\gamma'$, we define the metric as in \eqref{eqn:g_curv_homog}.
When $f = 0$, this metric is isometric to the product metric $\mathbb{H}^2 \times \R$ and the function $f$ warps the metric to make it irreducible.

The outline of the paper is as follows.
In Section~\ref{sec:preliminaries}, we give basic definitions and background.
In Section~\ref{sec:metric_on_Mc}, we describe the metric where $C \not= 0$, which also extends to describe metrics where $C = 0$ at some points.
In Section~\ref{sec:foliation}, we describe the foliation $\mathcal{F}$ of $M_{irred}$ by complete, totally geodesic flat planes.
In Section~\ref{sec:geodesic_foliations_H2}, we characterize the geodesic line foliations of $\mathbb{H}^2$ by the curves orthogonal to them.
In Section~\ref{sec:locally_irreducible}, consider the case where $M$ is locally irreducible at each point and prove Theorems~\ref{thm:continuous_coords} and~\ref{thm:irreducible_examples}.
In Section~\ref{sec:locally_reducible}, we describe the case where $M$ has locally split regions.
In Section~\ref{sec:topology} we prove Theorem~\ref{thm:pi1}.

The results in this paper include parts of the author's Ph.D. thesis under the direction of Dr. Wolfgang Ziller.
The atuhor is deeply grateful to Dr. Ziller for his invaluable guidance and patient support throughout the development and writing of these results.

\section{Preliminaries}
\label{sec:preliminaries}

In the remainder of the paper, we will assume that $M$ is complete and that the eigenvalues of the Ricci tensor are constants $(\lambda, \lambda, 0)$ with $\lambda = \pm 1$ unless otherwise stated.
Since the eigenvalues of the Ricci tensor determine the curvature tensor $R$ at any point up to an orthogonal transformation, we know that $R$ must be pointwise the curvature tensor of $\Sigma \times \R$ for $\Sigma$ either the round sphere or the hyperbolic plane.

Observe that this implies that at a point $p \in M$, there is a unit eigenvector $T$ of the zero eigenvalue of the Ricci tensor such that $\sec(T, X) = 0$ for all $X$ and $\sec(X,Y) = \lambda$ when $\{X,Y\}$ form a basis of $T^\perp \subset T_p M$.
Defining $\ker R_p := \{X \in T_p M | R(X, \cdot) \cdot = 0\}$ we get that $T$ spans $\ker R_p$.
We may define $T$ globally on $M$ by passing to a double cover if necessary.

It is well known that for any complete manifold, the distribution $\ker R$ has complete, totally geodesic leaves on the open subset where $\dim \ker R$ is minimal, see \cite{maltz}.
Hence the integral curves of $T$ are complete geodesics.
We call these the \emph{nullity geodesics} of $M$.
Thus we have that $T^\perp$ is a parallel distribution along each nullity geodesic.

Define $C: T^\perp \rightarrow T^\perp$ to be the splitting tensor of $T$, i.e. $C(X) = - \nabla_X T$.
Note that if $C$ vanishes in an open set, then by the de Rham splitting theorem, that set is locally isometric to a product $V \times \R$ with $V$ a surface of constant curvature $\lambda$.

Along a nullity geodesic $\gamma(t)$, we can choose a parallel basis $\curly{e_1, e_2}$ of $\ker R^\perp$.
Then $C$ written in this basis is a matrix $C(t)$ along $\gamma(t)$ satisfying
\begin{equation}
\label{eqn:ricatti}
C'(t) = C^2
\end{equation}
since
\begin{align*}
    0  &= R(X,T)T 
       = \nabla_T (C(X)) + C(\nabla_X T)
       = (\nabla_T C)(X) - C(C(X)).
\end{align*}

Note that \eqref{eqn:ricatti} has solutions $C(t) = C_0(I - tC_0)^{-1}$ for some matrix $C_0 = C(0)$.
Hence any real eigenvalue of $C$ must be zero.
Since $C$ is a $2 \times 2$ matrix, it either is nilpotent or has two non-zero complex eigenvalues.

Along a nullity geodesic,
\begin{equation*}
\Scal' = - 2 \trace C \cdot \Scal.
\end{equation*}
Indeed, from the second Bianchi identity
and the fact that in our case $\Scal = 2\chevron{R(X, Y) Y, X}$ for some orthonormal basis $\{X,Y\}$ of $T^\perp$,
\begin{align*}
    \Scal'
    &=  2\chevron{(\nabla_T R)(X, Y) Y, X}
    = 2\chevron{ R(Y, \nabla_{X} T) Y, X}
    + 2\chevron{R (\nabla_{Y} T, X) Y, X} \\
    &= - 2\Scal \cdot \chevron{C(X), X} - 2\Scal \cdot \chevron{C(Y), Y}
    = - 2 \trace C \cdot \Scal
\end{align*}
Since $M$ has constant scalar curvature, it follows that $\trace C$ is zero.
Note that \eqref{eqn:ricatti} also implies that $(\trace C)' = \trace (C^2) = (\trace C)^2 - 2 \det C$ along a nullity geodesic.
Hence $\det C = 0$ as well.
Since the only real eigenvalues of $C$ are zero, it follows that $C$ is nilpotent.

Define $M_C$ to be the subset of $M$ on which $C$ is non-zero.
Note that $M_{irred}$ is the closure of $M_C$ and that $M_{split}$ is the complement of $M_{irred}$, i.e. $M_{split}$ is the set of points $p \in M$ where $C = 0$ in a neighborhood of $p$.
Note that \eqref{eqn:ricatti} implies that if a nullity geodesic intersects $M_{split}$ ($M_C$ respectively) then it is contained in $M_{split}$ ($M_C$ respectively).
By \cite{graph_manifolds} Proposition 2.1, the universal cover of any connected component of $M_{split}$ is an isometric product $\Sigma \times \R$ where $\Sigma$ is a surface of constant curvature $\lambda$.

By going to a cover if necessary, we can define an orthonormal basis $e_1, e_2, T$ at any point in $M_C$ by $T \in \ker R$ and $e_1 \in \ker C$.
Since $C' = 0$, $e_1$ and $e_2$ are parallel along nullity geodesics.
There exists a smooth function $a \not= 0$ on $M_C$ so that $C(e_2) = a e_1$.
Hence
\begin{equation}
    \label{eqn:start_cov_derivs}
    \nabla_T e_1 = \nabla_T e_2 = \nabla_T T = 0, \quad \nabla_{e_1} T = 0, \quad \nabla_{e_2} T = - a e_1
\end{equation}
\begin{equation*}
    \nabla_{e_1} e_1 = \alpha e_2, \quad \nabla_{e_2} e_2 = \beta e_1, \quad \nabla_{e_1} e_2 = - \alpha e_1, \quad \nabla_{e_2} e_1 = a T - \beta e_2
\end{equation*}
for some smooth functions $\alpha, \beta$ on $M_C$.
Thus for the curvature tensor we have
\begin{align*}
    R(e_2, e_1) e_1 &= (e_1(\beta) + e_2(\alpha) - \alpha^2 - \beta^2) e_2
    + (a \beta - e_1(a)) T \\
    R(e_1, e_2) e_2 &= (e_1(\beta) + e_2(\alpha) - \alpha^2 - \beta^2) e_2
    + \alpha a T.
\end{align*}
Since $T \in \ker R$, this implies that
\begin{equation}
    \label{eqn:end_cov_derivs}
    \alpha = 0, \quad \Scal_M = e_1(\beta) - \beta^2, \quad \mbox{ and } \quad e_1(a) = a \beta.
\end{equation}

Again, since $T \in \ker R$, we have that $T(a) = T(\beta) = 0$, i.e. $a$ and $\beta$ are constant along nullity geodesics.
We let $D$ be the distribution on $M_C$ with $D_p$ spanned by $e_1, T \in T_p M$.
Note that \eqref{eqn:start_cov_derivs} implies that this distribution is completely integrable with totally geodesic leaves, which are flat since $T \in D_p$.
We denote by $\mathcal{F}_p$ the leaf of $D$ containing the point $p$.

\begin{lemma}
    \label{lemma:c_stays_nilpotent}
    Let $M^3$ be complete with constant Ricci eigenvalues $(\lambda, \lambda, 0)$ with $\lambda \not= 0$ and $M$ not everywhere locally reducible.
    Then
    \begin{enumerate}[(a)]
        \item up to scaling, $\lambda = -1$,
        \item integral curves of $e_1$ and $T$ starting at points in $M_C$ are complete geodesics contained in $M_C$, hence leafs $\mathcal{F}_p$ of the distribution $D$ are complete, and contained in $M_C$,
        \item and on $M_C$, we have $\abs{\beta} \leq 1$.
    \end{enumerate}
\end{lemma}

\begin{proof}
    Take $p \in M_C$.
    Then $e_1$ is well-defined in a neighborhood of $p$ in $M_C$.
    Since $\nabla_{e_1} e_1 = 0$, the integral curve of $e_1$ is a geodesic $\eta$ which is defined as long as $C \not= 0$.
    We first show that the complete geodesic $\eta$ lies in $M_C$.

    Writing a dot to indicate $e_1$ derivatives, we get
    \begin{equation*}
        \paren{\frac 1 a}^{\cdot \cdot} = - \paren{\frac {\dot a}{a^2} }^\cdot = - \paren{ \frac \beta a}^{\cdot} = - \frac{ (\Scal + \beta^2)} a + \frac {\beta^2} a.
    \end{equation*}
    Hence
    \begin{equation*}
        \paren{\frac 1 a}^{\cdot \cdot} + \frac 1 a \Scal = 0
    \end{equation*}
    and so $\tfrac 1 a$ satisfies the Jacobi equation.
    This equation holds only in $M_C$.
    We must then show that $a$ cannot go to zero along $\eta$.

    We can scale the metric so that $\lambda = +1$ or $\lambda = -1$.
    If $\Scal = 2$ is positive, then $\tfrac 1 a$ has solutions of the form
    $\frac 1 a = A_0 \cos t + A_1 \sin t$ which is bounded and hence $a$ never goes to zero.
    Therefore $\eta$ remains in $M_C$.
    But then there is a zero of $\frac 1 a$ in finite time which implies that $a$ diverges.
    This is a contradiction since $C$ is well-defined on all of $M$.
    Hence we may assume that $\lambda = -1$.

    Thus the solutions of \eqref{eqn:ricatti} are of the form $\tfrac 1 a = A_0 \cosh (t) + A_1 \sinh(t)$.
    Hence $a \rightarrow 0$ only as $t \rightarrow \pm \infty$ and therefore $C$ remains non-zero along $\eta$ for all time.
    Since \eqref{eqn:ricatti} implies that $C$ is constant along nullity geodesics as well, $C$ cannot go to zero on any leaf of the span of $\{e_1, T\}$ and hence the leaf is complete.

    Since $\beta = e_1(a)/a = -a e_1(1/a)$, we have that
    \begin{equation} \label{eqn:beta} \beta(t) = - \frac{A_0 \sinh(t) + A_1 \cosh(t)}{A_0 \cosh(t) + A_1 \sinh(t)} = - \frac{ \tanh(t) - \beta(0)}{ 1 - \beta(0) \tanh(t)}. \end{equation}
    This implies that $\abs{\beta} \leq 1$ since otherwise $\beta$ has a singularity in finite time along the complete geodesic $\eta$.
\end{proof}
Notice that  \eqref{eqn:beta} implies that if $\beta = \pm 1$ at any point, then it is $\pm 1$ along the entire nullity geodesic through that point.
Furthermore, if $M$ is simply connected, then the leafs of $D$ are isometric to $\R^2$ since $\exp$ is a diffeomorphism.

Finally, we observe that $a$ is a smooth function with $a = 0$ whenever $C = 0$.
To see this, let $\Theta$ be the rotation of $TM$ by $\pi/2$ about $T$ which takes $e_1$ to $e_2$ at points where $C \not= 0$.
This is smoothly defined (at least locally) on $M$ since $T$ is smooth.
Then $\Theta C$ has eigenvalues $a$ and $0$ since $\Theta C(e_2) = \Theta (a e_1) = a e_2$.
Hence the characterstic polynomial of $\Theta C$ is $t^2 - at$, and so $a$ is smooth on all of $M$.

\section{Description of the Metric on \texorpdfstring{$M_C$}{MC}}
\label{sec:metric_on_Mc}

The form of metrics with Ricci eigenvalues $(-1,-1,0)$ is well-known locally at points where $C \not= 0$ \cite{szabo2, KTV_curv_homog, KTV_new_examples}.
This form is a special case of the metric due to Sekigawa \cite{sekigawa} given by
\begin{equation}
\label{eqn:sekigawa_form}
g = p(x,u)^2 dx^2 + (du - v \; dx)^2 + (dv + u \; dx)^2.
\end{equation}
Specifically, if $M$ Ricci eigenvalues $(-1,-1,0)$, then $p(x,u) = f_1(x) \cosh(u) + f_2(x) \sinh(u)$.

We will work with a different parametrization (in the $x$ coordinate) which then enables us to include points where $C = 0$.
The metrics will be of the form of \eqref{eqn:g_curv_homog}, see Proposition~\ref{prop:g_curv_homog}.
Moreover, we will show that this form holds in a  ``global'' sense: that such a coordinate chart covers an entire connected component of $M_C$ when $M$ is simply connected and complete.
Choosing $f(x) = 0$ and $h(x) = 0$ makes the $u,v$ coordinates a standard parametrization of the product metric on $\mathbb{H}^2 \times \R$.
Moreover, if $f(x) = 0$ with any function $h$, the metric is locally isometric to $\mathbb{H}^2 \times \R$, and, as follows from the lemma below, is complete if $\abs{h(x)} \leq 1$ for all $x$.

We begin with two technical lemmas.
The first considers the basic properties, particularly completeness, of metrics which have the form of \eqref{eqn:g_curv_homog}.
\begin{lemma}
    \label{lemma:g_complete}
    Suppose $g$ is a metric on $V = (a_1,a_2) \times \R^2$, with coordinates $x \in (a_1,a_2)$ and $u,v \in \R^2$, of the form \eqref{eqn:g_curv_homog} with $f, h:(a_1, a_2) \rightarrow \R$ smooth.
    Then
    \begin{enumerate}[(a)]
        \item $g$ has Ricci eigenvalues $(-1,-1,0)$,
        \item $T = \pder{}{v}$ and $e_1 = \pder{}{u}$ and each leaf $\mathcal{F}_p$ is given by a plane with $x$ constant,
        \item $a(x,u,v) = f(x) (\cosh u - h(x) \sinh u)^{-1}$,
    \item if $f(x) \not= 0$ then $\beta = (h(x) \cosh u - \sinh u)(\cosh u - h(x) \sinh u)^{-1}$, and $e_2 = (\cosh u - h(x) \sinh u)^{-1} \paren{\pder{}{x} + v f(x) \pder{}{u} -  u f(x) \pder{}{v}}$,
        \item $g$ is locally irreducible if and only if $f^{-1}(0)$ contains no open subsets,
        \item $V$ is complete if and only if $(a_1,a_2) = (-\infty, \infty)$ and $\abs{h(x)} \leq 1$ for all $x$, and
        \item $h$ is the geodesic curvature of the path $u = v =0$.
    \end{enumerate}
\end{lemma}

\begin{proof}
    Parts (a-d) and (g) follow by direct computation.

    Note that (c) implies that if $a \not= 0$ at some point of an $x = const.$ plane then it is non-zero at every point on that plane.
    Thus (e) follows since $M$ is locally reducible at a point if and only if $C = 0$ on an open neighborhood.

    Now consider part (f).
    From (d) and Lemma~\ref{lemma:c_stays_nilpotent}, it follows that $\abs{h(x)} \leq 1$ is necessary for completeness.

    Define $A(x) = \int_0^x f(X) dX$ (see also Lemma~\ref{lemma:def_A} for a related function).
    We make a change of coordinates by
    \begin{equation*}
        (x,y,z) = (x, u \cos A(x) - v \sin A(x), u \sin A(x) + v \cos A(x)).
    \end{equation*}
    This performs a rotation in each $u$-$v$ plane by an amount that depends on $x$.
    In these coordinates, $g$ has the form
    \begin{equation*}
        g = p(x,y,z)^2 dx^2 + dy^2 + dz^2
    \end{equation*}
    where $p(x,y,z) = \cosh(u) - h(x) \sinh(u)$ with $u(x,y,z) = y \cos A(x) + z \sin A(x)$.
    This is an explicit form of the metric described in Theorem 2.5 of \cite{szabo2}.

    We now prove that $g$ is not complete if and only if the interval $(a_1,a_2)$ has $a_1$ or $a_2$ finite.
    Suppose that $g$ is not complete.
    Then there is a path $\gamma$ of finite length which has no limit in $V$.
    Let $\gamma(t) = (x(t), y(t), z(t))$.
    Then $\int \abs{y'(t)} dy$ and $\int \abs{z'(t)} dt$ are both lower bounds for the length of $\gamma$.
    Hence $y(t)$ and $z(t)$ are bounded.
    In particular, this shows that $\abs{u} \leq R$ for some $R \in \R$.
    Since $\abs{h(x)} \leq 1$,
    \begin{equation*}
        \abs{ \cosh(u) - h(x) \sinh(u) } \geq 2 e^{-R}
    \end{equation*}
    and $p(x,y,z) \geq 2e^{-R}$ at any point of $\gamma$.

    Since $\int \abs{p(x,y,z)} \abs{x'} dt$ is also a lower bound for the length of $\gamma$, $\int \abs{x'} dt$ is also  finite.
    So either $a_1$ or $a_2$ must be finite.

    For the other direction, either $a_1$ or $a_2$ is finite.
    Without loss of generality, we will assume that $a_2 \geq 0$ is finite and $a_1 \leq 0$.
    Then consider the path $\gamma(x) = (x,0,0)$ for $x \in [0,a_2)$.
    This path has length $\int_0^{a_2} dx$ which is finite but has no limit in $V$.
    Hence $V$ is not complete.
\end{proof}

Observe that in the metric \eqref{eqn:g_curv_homog}, if $f(x) = 0$ on an interval $I$, then $I \times \R^2 \subset M_{split}$.
Thus this description of the metric allows us to glue a split region to a non-split part of the metric.
This is the first example with this property.
Next we describe the metric on $M_C$.

\begin{prop}
    \label{prop:g_curv_homog}
    Suppose that $M^3$ is a complete, simply connected Riemannian manifold with Ricci eigenvalues $(-1, -1, 0)$.
    Then any connected component of $M_C$ has smooth coordinates $(x,u,v) \in (a_1, a_2) \times \R^2$ (with $a_i$ possibly $\pm \infty$) with metric of the form \eqref{eqn:g_curv_homog}
    for some smooth functions $f, h: (a_1, a_2) \rightarrow \R$ with $f(x) \not= 0$ and $\abs{h(x)} \leq 1$.
    The boundaries of this component are complete, flat, totally geodesic planes, one for each $a_i$ that is finite.
\end{prop}

\begin{proof}
    Recall that \eqref{eqn:g_curv_homog} is
    \begin{equation*}
        g = (\cosh(u) - h(x) \sinh(u))^2 dx^2 + (du - v f(x) \; dx)^2 + (dv + u f(x) \; dx)^2.
    \end{equation*}
    Let $V$ be a connected component of $M_C$.
    Fix a maximal integral curve $\gamma: (a_1, a_2) \rightarrow M$ of the vector field $e_2$ on $V$.
    By maximality, $a(\gamma(t)) \not= 0$ for $t \in (a_1, a_2)$ and $\lim_{t \rightarrow a_i} a(\gamma(t)) = 0$ if $a_i$ is finite.
    Let $N$ be the manifold defined by one coordinate chart with $(x,u,v) \in (a_1, a_2) \times \R^2$
    and metric of the form in \eqref{eqn:g_curv_homog} where $f(x) = a(\gamma(x))$ and $h(x) = -\beta(\gamma(x))$,
    By Lemma~\ref{lemma:c_stays_nilpotent} this implies that $\abs{h} \leq 1$.

    The manifold $N$ is simply connected (but may not be complete).
    Define $\phi: N \rightarrow M$ by
    \begin{equation*}
        \phi(x,u,v) = \exp_{\gamma(x)} (u e_1 + v T).
    \end{equation*}
    Notice that by Lemma~\ref{lemma:c_stays_nilpotent}, $\phi(N) \subset M_C$ and we will show that $\phi$ is in fact an isometry onto $V$.

    We first show that $\phi$ is a local isometry.
    Note that $\pder{\phi}{u} = e_1$ and $\pder{\phi}{v} = T$.
    We next compute $\pder{\phi}{x}$.
    Fix $(x_0, u_0, v_0)$.
    Consider the family of geodesics $\alpha_s(t) = \phi(x_0 + s, t u_0, t v_0)$.
    Define $J(t)$ to be the Jacobi field along $\alpha_0$ corresponding to the variation $\alpha$.
    Then 
    \begin{equation}
        J(0) = \gamma'(0) =  e_2, J'(0) = \frac{D}{ds} \pder{\alpha_s}{t}\big|_{s=0,t=0} = \nabla_{e_2} (u_0 e_1 + v_0 T) = u_0 (a T - \beta e_2) - v_0 a e_1.
    \end{equation}
    It follows from \eqref{eqn:start_cov_derivs}--\eqref{eqn:end_cov_derivs} that the Jacobi field with these initial conditions is given by
    \begin{equation*}
    J(t) =
        - v_0 a(\gamma(x_0)) t e_1
        + \brak{ \cosh(u_0 t) - \beta(\gamma(x_0)) \sinh(u_0 t) } e_2
        + u_0 a(\gamma(x_0)) t T.
    \end{equation*}
    Since $f(x) = a(\gamma(x))$ and $h(x) = \beta(\gamma(x))$, we see that
    \begin{equation*}
        \left.\pder{\phi}{x}\right|_{(x_0,u_0,v_0)} = J(1) = \paren{\cosh u_0 - h(x_0) \sinh u_0} e_2 - v_0 f(x) e_1 + u_0 f(x) T.
    \end{equation*}
    Now it is easy to compute that $\phi$ is a local isometry.

    We now prove that $\phi$ is a covering map $N \rightarrow \phi(N)$.
    It suffices to show that $\phi$ has the path lifting property.
    Let $\mu: [0,1] \rightarrow \phi(N) \subset V$ be a path.
    Then there exists $\delta >0$ such that $\abs{a(\mu(t))} \geq \delta$ for all $t$.
    If $\tilde \mu: [0, t_0) \rightarrow N$ is a lift of $\mu$, then $\abs{a(\tilde \mu|_{[0,t_0)})} \geq \delta$ as well since $a$ is, up to sign, an isometry invariant.
    Thus $x(\tilde \mu|_{[0, t_0)}) \in [b_1, b_2] \subset (a_1, a_2)$ for some $[b_1, b_2]$.
    Since $\mbox{length}(\tilde \mu|_{[0,t]}) = \mbox{length}(\mu|_{[0,t]})$ for all $t < t_0$, the $u, v$ coordinates along $\tilde \mu$ are bounded as well, and hence $\lim \tilde \mu$ lies in compact set which implies that $\mu$ can be lifted past $t = t_0$.

    The same argument implies that $\phi(N) = V$ since for any point $p \in V$, we can choose a path from $p$ to $p^* \in \phi(N)$, which by the above argument has a lift.

    In order to show that $\phi$ is injective, let $\tilde \gamma$ be the integral curve of $e_2$ through a point $\tilde p \in N$.
    Since $\phi$ takes each leaf of $\mathcal{F}$ in $N$ isometrically to a leaf of $\mathcal{F}$ in $M$, it suffices to show that $\phi \circ \tilde \gamma$ intersects each leaf of $\mathcal{F}$ at most once.
    Let $\mathcal{F}_0$ the leaf through $p = \phi(\tilde p)$.
    Since $\sec \leq 0$ and $M$ is simply connected, we have a globally defined signed distance function $t: M \rightarrow \R$ to the leaf $\mathcal{F}_0$.
    The integral curves of $\grad t$ are the geodesics orthogonal to $\mathcal{F}_0$.
    Hence $\frac{d}{ds}( t \circ \gamma(s)) = \chevron{\grad t, \gamma'} \not= 0$ since otherwise $\grad t \in T \mathcal{F}_{\gamma(s)}$ and so $\mathcal{F}_{\gamma(s)}$ intersects $\mathcal{F}_0$.
    Thus $t$ is monotonic on $\gamma$ which implies that $\phi$ is injective and hence an isometry onto $V$.

\end{proof}

\section{Foliation by Flat Planes}
\label{sec:foliation}

We now discuss the properties of the foliation $\mathcal{F}$ on $M_C$.
We will see that it extends to a Lipschitz foliation on the closure $M_{irred}$ of $M_C$.
Furthermore, there are $C^{1,1}$ curves everywhere orthogonal to the foliation, and that the connected components of $M_{irred}$ are plane bundles over these curves.
We assume until Section~\ref{sec:topology} that $M$ is simply connected.

\begin{lemma}
    \label{lemma:cont_foliation}
    Let $M^3$ be a complete, simply connected Riemannian manifold with Ricci eigenvalues $(-1, -1, 0)$.
    Then $\mathcal{F}$ extends to a continuous foliation on $M_{irred}$ whose leaves are complete, flat, totally geodesic planes.
\end{lemma}

\begin{proof}
    By Lemma~\ref{lemma:c_stays_nilpotent}, through every point $p \in M_C$ there is a complete, flat, totally geodesic leaf $L_x$, which is given by $\exp_{p}(u e_1 + v T)$ for $u,v \in \R$.
    Consider a sequence of points $p_k \rightarrow p$ with $p_k \in M_C$ and $p \not\in M_C$.
    Since $T$ is smooth on all of $M$, $T_{p_k} \rightarrow T_{p}$.

    Next, suppose that $(e_1)_{p_k}$ does not converge to a unit vector at $p$.
    Then there must be two subsequences $q_k \rightarrow p$ and $r_k \rightarrow p$ with $(e_1)_{q_k} \rightarrow X$ and $(e_1)_{r_k} \rightarrow Y$ with $X \not= \pm Y$.
    Since $e_1$ is always perpendicular to $T$, then $X, Y$ are orthogonal to $T$.
    Defining $Q = \curly{\exp_{p}(uX + vT)| u,v \in \R}$ and $R = \curly{\exp_{p}(uY + vT)| u,v \in \R}$, then $\mathcal{F}_{q_k} \rightarrow Q$ and $\mathcal{F}_{r_k} \rightarrow R$.
    So $Q$ and $R$ intersect at $p$ and each separate $M$ into two halves.
    Then points for large $k$, $\mathcal{F}_{q_k}$ and $\mathcal{F}_{r_k}$ intersect and hence are equal.
    This contradicts $X \not= \pm Y$.
    Therefore $e_1, e_2$ and $\mathcal{F}$ extend to $M_{irred}$.

\end{proof}

In \cite{zeghib}, it is shown that every codimension one geodesic foliation of a smooth (not necessarily complete) manifold is locally Lipschitz.
Due to this, the Picard-Lindel\"of existence and uniqueness theorem for ODEs implies that there exists a unique $C^{1,1}$ curve orthogonal to the foliation through each point.
We apply this to our foliation $\mathcal{F}$ of $M_{irred}$.
\begin{prop}
\label{prop:plane_bundle}
    Suppose that $M$ is a complete, simply connected $3$-manifold with Ricci eigenvalues $(-1, -1, 0)$.
    For any point $p \in M_{irred}$, there exists a unique, maximal (in $M_{irred}$) $C^{1,1}$ integral curve $\gamma$ of $e_2$ which is orthogonal to $\mathcal{F}$ at every point.
    Furthermore, $\gamma$ intersects exactly once each leaf of $\mathcal{F}$ in the connected component of $M_{irred}$ containing $p$.
\end{prop}

\begin{proof}

    Consider the maximal integral curve $\gamma$ of $e_2$ at some point $p \in M_{irred}$.
    We can assume that $\gamma$ is maximal in the connected component $V$ of $M_{irred}$ that contains $p$.
    Since $\gamma$ has unit speed, the domain of $\gamma: I \rightarrow V$ is a closed interval $I$ (possibly infinite or half-infinite).
    Define $\exp^\perp_{\gamma}$ to be $\exp$ restricted to the subset of $TM$ where $(x,V) \in TM$ is such that $x = \gamma(t)$ and $V$ is perpendicular to $\gamma'(t)$ for some $t$.
    We claim that $\exp^\perp_{\gamma}$ is onto $V$, i.e. that $\gamma$ intersects each leaf of $\mathcal{F}$ in $V$ once.
    We first prove that $\im(\exp^\perp_{\gamma})$ is closed.
    If not, then there exists a point $q \in V \setminus \im(\exp^\perp_\gamma)$ and a sequence of points $q_k \in \im(\exp^\perp_\gamma)$ with $q_k \rightarrow q$.
    Then $\mathcal{F}_{q_k} \rightarrow \mathcal{F}_q$.
    For each $q_k$, let $\gamma(t_k)$ be a point on $\gamma$ through $\mathcal{F}_{q_k}$.
    Since $I$ is closed, if $t_k$ is bounded, then the $t_k$ have a limit point $t_*$ in $I$, which implies that $\mathcal{F}_{t_*} = \mathcal{F}_q$, which is a contradiction.
    So we may assume that $t_k \rightarrow \infty$.
    Then $\mathcal{F}_{\gamma(t)} \rightarrow \mathcal{F}_q$ as $t \rightarrow \infty$.

    Let $\eta_t$ be the shortest path from $\gamma(t)$ to $\mathcal{F}_q$ and $y(t)$ be the length of $\eta_t$.
    Note that since $\mathcal{F}$ is Lipschitz, we have that $\chevron{e_2, \eta_t'} \geq 1 - c y(t)$ for some constant $c$, for any $\gamma(t)$ sufficiently close to $\mathcal{F}_q$.
    Considering the variation of geodesics $\eta_x$, the first arc-length variation formula shows that
    \begin{equation*}
        \frac{d}{dt} y = -\chevron{\gamma', \eta_t'} = - \chevron{e_2, \eta_t'} \leq -1 + c y.
    \end{equation*}
    When $0 < y < 1/(2c)$, it follows that $\frac{d}{dt} y < - 1/2$, and hence $y(t) \rightarrow 0$ in finite time.
    Since $I$ is closed, we again get a contradiction that $\mathcal{F}_q$ must be in $\exp^\perp_\gamma$.
    Hence $\im(\exp^\perp_\gamma)$ is closed.

    Now we want to show that $\im(\exp^\perp_\gamma)$ is all of $V$.

    First we argue that $V$ is convex.
    Suppose that there exists a geodesic $\mu:[a,b] \rightarrow \R$ with endpoints in $V$ but completely not contained in $V$.
    Then there exists a $t_0$ such that $\mu(t_0) \not \in V$.

    Since $V$ is closed, there exists a leaf $P \in \mathcal{F}$ which is the last leaf of $\mathcal{F}$ before $\mu(t_0)$.
    Define $U$ to be a subset of the unit vectors at $\mu(t_0)$ by
    \begin{equation*}
        U := \{X \in T_{\mu(t_0)}^1 M | \exp_{\mu(t_0)}( t X) \in P \mbox{ for some } t > 0 \}.
    \end{equation*}
    Note that $U$ is connected, and non-empty.
    It is open since $\exp_{\mu(t_0)}(t X)$ for $X \in U$ must be transverse to $P$ since otherwise the fact that $P$ is totally geodesic would imply that $\mu(t_0) \in P$.

    Let $U_M =  \curly{\exp(s v) | v \in U, s > 0 }$, which is open.
    Then its boundary 
    \[ \partial U_M = \curly{\exp(s v) | v \in \partial U, t \geq 0}. \]
    Suppose that $\partial U_M$ is not disjoint from $V$.
    Then there is a $Q \in \mathcal{F}$ intersecting $\partial U_M$.
    Since $Q$ is totally geodesic and does not contain $\mu(t_0)$, $Q$ must be transverse to $\partial U_M$.
    So $Q$ also intersects $U_M$.
    Then there is a geodesic from $\mu(t_0)$ to a point on $P$ that intersects $Q$ transversely.
    Since a geodesic in a $\sec \leq 0$ space can only cross a transverse geodesic hyperplane once, $Q$ must separate $\mu(t_0)$ from a point on $P$.
    Therefore it separates $\mu(t_0)$ from all of $P$ since $Q$ and $P$ are disjoint.
    This contradicts the fact that no plane of $V$ lies on $\mu$ between $P$ and $\mu(t_0)$.

    Therefore $U_M$ is an open subset of $M$ whose boundary does not interesct $V$.
    Then $V \cap U_M$ and $V \cap (M \setminus U_M)$ are two disjoint open sets covering $V$ which is a contradiction with $V$ being connected.
    Therefore $\im(\exp^\perp_\gamma)$ must be convex.

    Now we show that $\im(\exp^\perp_\gamma)$ is onto $V$.
    Suppose there is $x \in V$ and $x \not\in \im(\exp^\perp_\gamma)$.
    Then the geodesic from $\gamma(0)$ to $x$ stays in $V$ and let $L$ be the last leaf of $\mathcal{F}$ in $\im(\exp^\perp_\gamma)$ it passes through.
    Then $L = \exp^\perp_{\gamma(t_0)}$ for some $t_0$.
    This contradicts maximality of $\gamma$ since $L$ (and hence $\gamma(t_0)$) must be in the interior of $V$.

    We will now see that $\gamma$ intersects each leaf of the foliation at most once.
    Take $L_0$ to be the leaf of $\mathcal{F}$ through $\gamma(t_0)$, for some time $t_0$.
    Suppose that $\gamma$ intersects $L_0$ at some time $t_1 > t_0$.
    Let $t_*$ be the time in $[t_0, t_1]$ where $\gamma$ is maximally far from $L_0$.
    Then $\gamma$ must be orthogonal to the shortest geodesic $\eta$ from $L_0$ to $\gamma(t_*)$.
    Since $\gamma$ is orthogonal to the leaves of $\mathcal{F}$, the totally geodesic leaf $L_*$ through $\gamma(t_*)$ must contain the geodesic $\eta$.
    Note that $L_* \not= L_0$ since $\gamma(t_*)$ is maximally far from $L_0$ and $\gamma$ is orthogonal to $L_0$ at $t_0$.
    But $L_*$ intersects $L_0$, a contradiction.

    Hence $\gamma$ intersects each leaf of $\mathcal{F}$ exactly once.
\end{proof}

\section{Geodesic Foliations of \texorpdfstring{$\mathbb{H}^2$}{H2}}
\label{sec:geodesic_foliations_H2}
We now consider a lower-dimensional analog of $\mathcal{F}$, namely foliations of $\mathbb{H}^2$ by complete geodesics.
This will be used in Theorem~\ref{thm:irreducible_examples}.
Again, these foliations are Lipschitz by \cite{zeghib} and have $C^{1,1}$ curves orthogonal to them.
We study some basic properties of these curves.

Let $H$ be the turning angle of $\gamma$, i.e. the angle between $\gamma'(x)$ and $V$ where $V$ is a parallel translation of $\gamma'(0)$ along $\gamma$.
The following lemma about turning angles is no doubt well-known, but we provide a proof for completeness since we were unable to find a reference for it.

\begin{lemma}
    Suppose that $H: \R \rightarrow \R$ is locally Lipschitz and $\Sigma$ is a complete surface.
    For any starting point $p_0$ and initial unit vector $v_0$, there exists a unique arc-length parametrized $C^{1,1}$ curve $\gamma: \R \rightarrow \Sigma$ whose turning angle at $\gamma(t)$ is $H(t)$.
    \label{lemma:turning_angle_curve}
\end{lemma}
\begin{proof}
    Choose local coordinates $\vec{x}$ of a neighborhood $U \subset \Sigma$ of $p_0$.
    We proceed by modifying the standard argument for the existence of the geodesic flow on $T U$.
    Choose coordinates $(\vec{x}, \vec{y})$ on $T U$ such that $\vec{y} = \sum_{i=1}^2 y_i \pder{}{x_i}$.
    For two vectors $\vec{y}, \vec{z}$ in $T_{\vec{x}} U$, define $\vec{v}_{\vec{x}}(\vec{y},\vec{z})= - \sum_{i,j} \Gamma^{k}_{ij} y_i z_j \pder{}{y_k}$ where $\Gamma^{k}_{ij}$ are the Christoffel symbols at $\vec{x}$.

    Define $\Theta_{\vec{x}, r}: T_{\vec{x}}U \rightarrow T_{\vec{x}}U$ to be the rotation of of $T_\vec{x} U$ by angle $r$ at each $\vec{x}$ and $W$, a time-dependent vector field on $T U$ given by
    \begin{equation*}
        W(\vec{x}, \vec{y}, t) = (\Theta_{\vec{x}, H(t)}(\vec{y}), \vec{v}_{\vec{x}}(\vec{y}, \Theta_{H(t)}(\vec{y})).
    \end{equation*}

    Then the ODE defined by $\frac{d}{dt} (\vec{x}, \vec{y}) = W(\vec{x}, \vec{y}, t)$ is continuous in $t$ and smooth in $(\vec{x}, \vec{y})$.
    By the standard Picard-Lindel\"of theorem, there exists a unique solution $(\gamma(t), Y(t))$ where $\gamma$ is a $C^{1,1}$ curve and $Y$ a vector field along that curve.
    We choose initial conditions so that $\gamma(0) = p_0$ and $Y(0) = \gamma'(0) = v_0$.

    Now we compute $\nabla_{\gamma'} Y$.
    Writing $Y = (y_1, y_2)$ and $\gamma' = (z_1, z_2)$ we get that
    \begin{align*}
        \nabla_{\gamma'} Y
            &= \sum_{k} \paren{\sum_{ij} \Gamma^k_{ij} y_i z_j + \gamma'(y_k)} \pder{}{x_k} \\
            &= -\vec{v}_{\vec{x}}(Y, \gamma'(t)) + \sum_{k} \frac{d}{dt}(y_k) \pder{}{x_k}\\
            &= -\vec{v}_{\vec{x}}(\gamma'(t), Y) + \vec{v}_{\vec{x}}(\gamma'(t), Y) = 0
    \end{align*}
    so $Y$ is parallel along $\gamma(t)$ and $\norm{\gamma'} = 1$.
    By the ODE, $\gamma' = \Theta_{H(t)}(Y)$, and $\gamma$ has turning angle $H(t)$.

    Lastly, note that Picard-Lindel\"of gives existence of $\gamma$ and $Y$ for at least time $1/C$ where $C$ is the maximum derivative of components of $W$ on $U$.
    For $R > 0$, there exists such a $C$ on the ball $B_{2R}$ centered at $\gamma(0)$.
    Therefore we may repeatedly extend the existence of $\gamma$ until either it exists for at least time $R$ or it has left the ball $B_R$.
    Since $\gamma$ has unit speed, the second case cannot occur without also existing until at least time $R$.
    Taking $R \rightarrow \infty$, $\gamma$ exists for all time.
\end{proof}

Suppose that $\gamma: \R \rightarrow \mathbb{H}^2$ is a $C^{1,1}$ curve that is arc-length parameterized.
Let $X$ be a unit vector field along $\gamma$ that is perpendicular to $\gamma'$ everywhere.
Define $\exp^\perp: \R^2 \rightarrow \mathbb{H}^2$ by
\begin{equation*}
\exp^\perp(s,t) = \exp_{\gamma(s)}(t X).
\end{equation*}

\begin{prop}
    \label{prop:foliation_H_lipschitz}
    Fix a point $p \in \mathbb{H}^2$ and a vector $V \in T_p \mathbb{H}^2$.
    There is a bijection between Lipschitz functions $H: \R \rightarrow \R$ with Lipschitz constant $1$ and arc-length parameterized curves $\gamma$  such that
    \begin{enumerate}[(a)]
        \item $\gamma(0) = p$,
        \item $\gamma'(0) = V$, and
        \item the curves $\eta_s: t \mapsto exp^\perp(s,t)$ form a foliation of $\mathbb{H}^2$.
    \end{enumerate}
    In particular, $H$ is the turning angle of the corresponding $\gamma$.
\end{prop}

\begin{proof}
    Lemma~\ref{lemma:turning_angle_curve} shows that $\gamma$ is determined uniquely by its starting conditions and turning angle $H$.

    Next, we assume that $\gamma$ satisfies (a) and (b) and has turning angle $H$.
    To see that (c) holds, it suffices to show that each point $p \in \mathbb{H}^2$ has a unique closest point $q$ on $\gamma$.
    Then $p$ lies on the orthogonal geodesic at $q$ and uniqueness implies that the orthogonal geodesics are all disjoint, and hence foliate $\mathbb{H}^2$.

    Let $\delta(q) = d(p,q)$ be the distance function to $p$.
    By standard hyperbolic trigonometry,
    \[ \nabla_{X} \grad \delta = \coth(\delta) \chevron{X, \grad \delta^\perp} \grad \delta^\perp \]
    where $\grad \delta^\perp$ is a unit vector orthogonal to $\grad \delta$.
    Note that almost everywhere $\nabla_{\gamma'} \gamma' = h(t) (\gamma')^\perp$ where $h = H'$ and $(\gamma')^\perp$ is a unit vector orthogonal to $\gamma'$.

    The case where $\gamma$ is smooth is shown in \cite{ferus_geod_foliations}, and we follow the same strategy.
    Define $L(q) = \cosh(\delta(q))$.
    Since $H$ is Lipschitz, $\gamma$ is in-fact $C^{1,1}$ and is twice-differentiable almost everywhere.
    So $L \circ \gamma$ is twice differentiable a.e. and, therefore 
    \begin{align*}
    (L \circ \gamma)'  &= \chevron{\grad L, \gamma'} = \sinh(\delta) \chevron{ \grad \delta, \gamma'},\\
    (L \circ \gamma)'' 
        &= \cosh(\delta) \chevron{\grad \delta, \gamma'}^2 
                + \sinh(\delta) \brak{
                    \coth(\delta) \chevron{\grad \delta^\perp, \gamma'}^2
                    + \chevron{\grad \delta, h(t) (\gamma')^{\perp}}
                    } \\
        &= \cosh(\delta) + h(t) \sinh(\delta) \chevron{\grad \delta, (\gamma')^{\perp}}.
    \end{align*}

    Since the Lipschitz constant of $H$ is 1,  $\abs{h} \leq 1$, and therefore $(L \circ \gamma)'' \geq e^{-\delta} > 0$ a.e..
    Therefore $L$ is convex and has a unique minimum.
    Since $L$ is monotone in $\delta$, $\delta$ too has a unique minimum and (c) holds.

    For the converse, we assume that $\gamma$ has turning angle $H$ that does not have Lipschitz constant 1.
    In \cite{zeghib} it is shown that any co-dimension one geodesic foliation of a smooth manifold is locally Lipschitz.
    So $H$ is differentiable a.e.
    If $H$ does not have Lipschitz constant 1, then there is a point $x$ where $\abs{h(x)} > 1$, where $h = H'$.
    Then it is possible to choose $\epsilon$ such that $\cosh \epsilon + h \sinh \epsilon = 0$.
    Note that the Jacobi field of geodesics orthogonal to $\gamma$ is $(\cosh t + h(x) \sinh(t)) X$ (where $X$ is the parallel transport of $\gamma'$ along the orthogonal geodesic).
    Then  $\gamma$ has a focal point at distance $\epsilon$ from $\gamma(x)$ and so the orthogonal geodesics do not foliate $\mathbb{H}^2$.
\end{proof}

Recall that the geodesic curvature of a curve is the derivative of its turning angle $H(t)$.
\begin{corollary}
    If the orthogonal geodesics of $\gamma$ foliate $\mathbb{H}^2$, then $\gamma$ is $C^{1,1}$ and its geodesic curvature $h$ satisfies $\abs{h} \leq 1$ almost everywhere.
    \label{cor:foliating_curves}
\end{corollary}

\section{Locally Irreducible Metrics}
\label{sec:locally_irreducible}

In this section, we consider the case where $M$ has Ricci eigenvalues $(-1, -1, 0)$ and is simply connected and locally irreducible, i.e. $M = M_{irred}$.

We first show that the metric has the form as desired in  Theorem~\ref{thm:continuous_coords}.
\begin{proof}[Proof of Theorem~\ref{thm:continuous_coords}]
By Proposition~\ref{prop:plane_bundle}, there exists a maximal, unit speed $C^{1,1}$ curve $\gamma: \R \rightarrow M$ everywhere orthogonal to $\mathcal{F}$ such that each point of $M$ lies on $\mathcal{F}_{\gamma(t)}$ for some $t$.
Define $f(x) := a(\gamma(x))$ and $h(x) = \beta(\gamma(x))$, where $h(x)$ is defined only on the set $\{x \in \R: f(x) \not= 0\}$.

Define $x(p)$ such that $\mathcal{F}_{\gamma(x)} = \mathcal{F}_p$, and $u(p),v(p)$ such that $p = \exp_{\gamma(x)} (u e_1 + v T)$.
Then $(x,u,v)$ are smooth coordinates on each connected component of $M_C$ and they are Lipschitz since $\mathcal{F}$ is Lipschitz.

By Proposition~\ref{prop:g_curv_homog}, the metric on $M_C$ has the desired form and since $\overline{M_C} = M$, the theorem follows.
\end{proof}

\begin{corollary}
\label{cor:smooth_foliation}
If $\mathcal{F}$ is a smooth foliation, then $f, h$ are also smooth.
\end{corollary}

An immediate corollary of these results is Corollary~\ref{cor:analytic}, which gives a classification of the case where $M$ is irreducible and analytic.
Since $M$ is analytic, $C$ is analytic and hence $\mathcal{F}$ is analytic on $M$.
Hence $T$, $e_1$ and $e_2$ are analytic and so too are $f, h$.

We will next present some examples of Theorem~\ref{thm:irreducible_examples} and then give its proof.

\begin{example}
    \label{example:irreducible_1}
    Suppose $H$ is as in Theorem~\ref{thm:irreducible_examples} and is smooth.
    Then the metric $g$ is just given by \eqref{eqn:g_curv_homog}.
    Let $\gamma$ be the path in $\mathbb{H}^2$ with turning angle $H$ from Proposition~\ref{prop:foliation_H_lipschitz}.
    Then the orthogonal curves of $\gamma$ foliation $\mathbb{H}^2$ and we can put $g$ on $\mathbb{H}^2 \times \R$  using coordinates with $(x,u,v)$ is just $(\exp_{\gamma(x)} (u (\gamma')^\perp), v)$.
    In the case where $f = 0$, this $g$ becomes the standard metric on $\mathbb{H}^2 \times \R$.
    This gives the strategy for the proof of the theorem even when $H$ is not smooth.
\end{example}

\begin{example}
    \label{example:irreducible_2}
    Each $H$ corresponds, by Proposition \ref{prop:foliation_H_lipschitz}, to a $C^{1,1}$ curve $\gamma$ in $\mathbb{H}^2$ and $h$ is the geodesic curvature of $\gamma$.
    In Figure~\ref{fig:irreducible_examples}, we see three examples of curves $\gamma$ with their corresponding orthogonal geodesics, in the Poincar\'e disk model of $\mathbb{H}^2$.
    The first two are smooth curves, with $H(x) = 0$ and $H(x) = x$, so they have $h(x) = 0$ and $h(x) = 1$, respectively.
    The third curve is only $C^{1,1}$ and has one non-smooth point $\gamma(0)$.
    On the left half, it has $h(x) = 1$ and on the right $h(x) = -1$.
    Any choice of smooth $f(x)$ works for the first two examples.
    For the last example, any smooth $f(x)$ works as long as $f^{(k)}(0) = 0$ for all $k$.
    Therefore this demonstrates that $h$ need not even be continuous.
\end{example}

\begin{example}
    \label{example:irreducible_4}
    One may also choose $H$ such that $h$ is non-smooth on a Cantor set.
    Consider the trinary expansion of numbers in $[0,1]$.
    Define $h(t)$ to be $0$ if $1$ never occurs in the expansion, and $(-1)^n$ if the first $1$ occurs in the $n$th digit.
    Defining $h(t) = 0$ outside of $[0,1]$, $h(t)$ has discontinuities at a Cantor set.
    As $h(t)$ is the difference of two indicator functions, it is Lebesgue integrable.
    Let $H(t) = \int_0^t h(s) ds$.
    Then take any choice of $f$ which goes to zero to infinite order on the Cantor set and is nonzero elsewhere.
    See Figure~\ref{fig:cantor_set_example} for the corresponding foliation in $\mathbb{H}^2$.
\end{example}

\begin{proof}[Proof of Theorem~\ref{thm:irreducible_examples}]
    First, we apply Lemma~\ref{lemma:turning_angle_curve} to get a $C^{1,1}$ curve $\gamma$ in $\mathbb{H}^2$ with turning angle $H$.

    We proceed by defining $g_f$ a smooth symmetric tensor on $M = \mathbb{H}^2 \times \R$ such that $g = g_{\mathbb{H}^2 \times \R} + g_f$ is the desired metric, where $g_{\mathbb{H}^2 \times \R}$ is the product metric.
    Note that we can embed $\gamma$ into $M$ by $(\gamma(x),0)$, and we call this embedding $\gamma$ as well for simplicity.
    Let $e_1$ be a unit vector field along $\gamma$ which is orthogonal to $\gamma'$ in $\mathbb{H}^2$, and let $e_3$ be a unit vector field in the $\R$ factor of $M$.
    There are $C^0$ coordinates $(x,u,v)$ of $M$ such that  $p \in M$ has coordinates $(x,u,v)$ if $p = \exp_{\gamma(x)}( u e_1 + v e_3)$.
    Define $e_2$ to be a unit vector field  orthogonal to $\{e_1, e_3\}$.

    Let $S \subset \R$ be the set of $x$ values such that $\gamma$ is locally smooth at $\gamma(x)$.
    Then there is a subset $S_M \subset M$ of points $p$ such that $x(p) \in S$.
    $S_M$ is the set of points where the $(x,u,v)$ coordinates are locally smooth.

    Note that $(M,g_{\mathbb{H}^2 \times \R})$ has Ricci eigenvalues $(-1,-1,0)$ with $C = 0$ and on $S_M$ has a smooth foliation by complete totally geodesic planes and hence the metric is of the form \eqref{eqn:g_curv_homog} on $S_M$ by Proposition~\ref{prop:g_curv_homog}.
    Therefore, the vector $e_3 = T$ and $\{e_1,e_2,T\}$ satisfy all the equations of \eqref{eqn:start_cov_derivs}-\eqref{eqn:end_cov_derivs} when taking covariant derivatives with the Levi-Civita connection of $g_{\mathbb{H}^2 \times \R}$ where $f(x) = 0$ in those equations and $h(x) := \chevron{\nabla_{\gamma'} \gamma', e_1}$.
    Moreover, the contents of Lemma~\ref{lemma:g_complete} also apply in $S_M$.
    (Our goal is to modify $g_{\mathbb{H}^2 \times \R}$ to make our choice of $f(x)$ the one that occurs in these covariant derivatives.)

    \begin{figure}[h]
    \centering
    \includegraphics[scale=0.5]{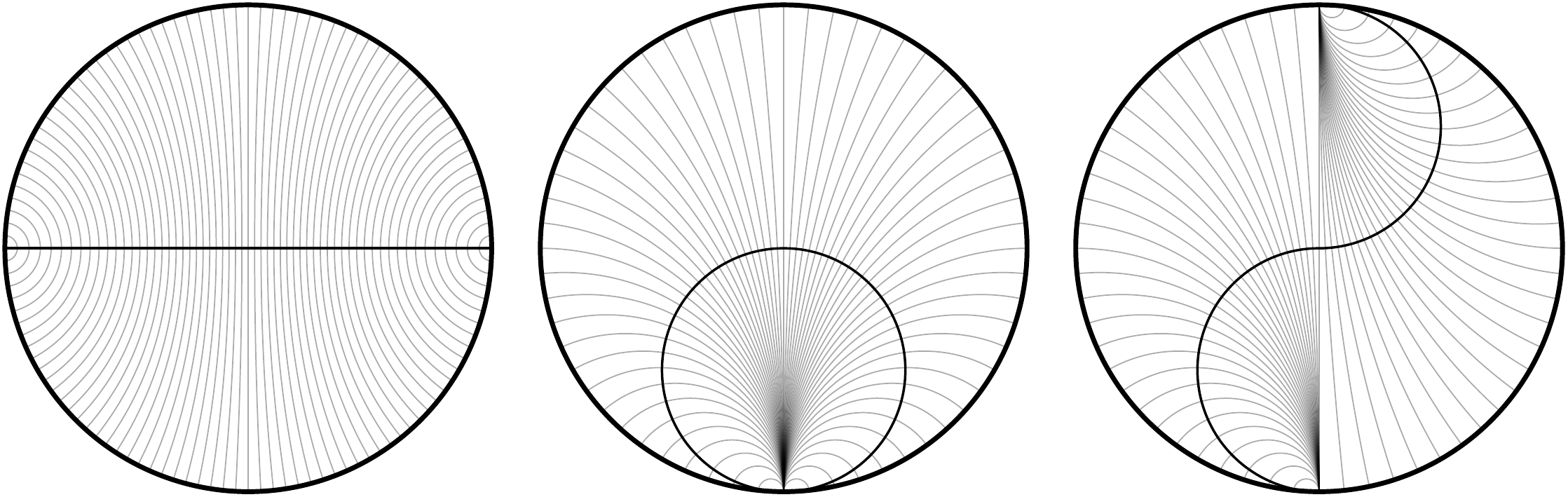}
    \caption{Three examples of possible choices of $H$ for Example~\ref{example:irreducible_2}, given by their corresponding paths in the Poincar\'e disk model of $\mathbb{H}^2$ along with their orthogonal, geodesic foliation.
    Note that the third example demonstrates that $h$ may be non-continuous.}
    \label{fig:irreducible_examples}
    \end{figure}

    \begin{figure}[h]
    \centering
    \includegraphics[scale=0.3]{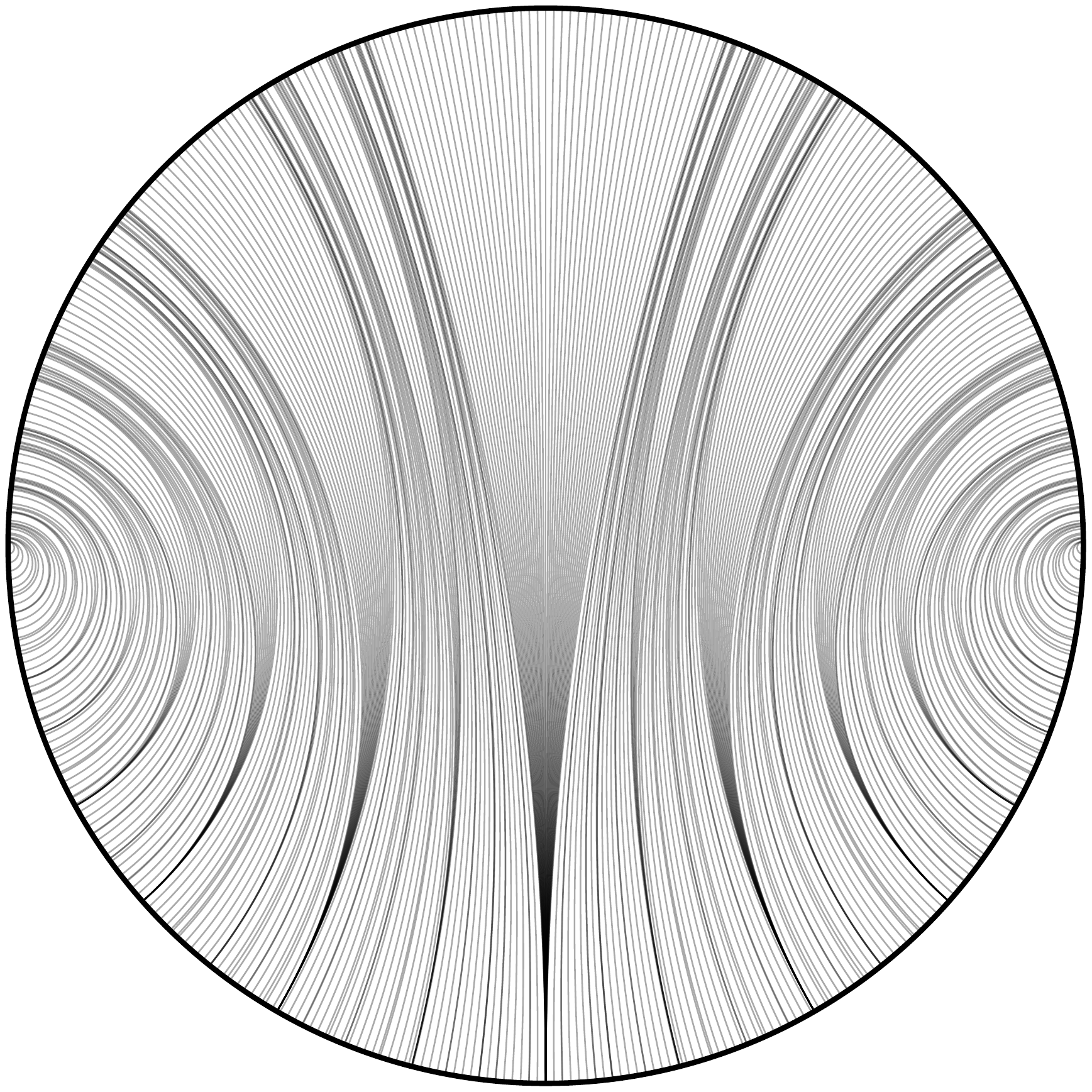}
        \caption{Geodesic foliations of $\mathbb{H}^2$ where the set of discontinuities of $h$ forms a Cantor set.}
    \label{fig:cantor_set_example}
    \end{figure}

    Define the symmetric 2-tensor $g_f$ point-wise on $M$ at points $p \in S_M$ by
    \begin{equation*}
        g_f = -2 f(x) v (dx \; du + du \; dx) + 2 f(x) u (dx \; dv + dv \; dx) + f(x)^2 (u^2 + v^2) dx^2
    \end{equation*}
    and by $g_f = 0$ for $p \not\in S_M$.
    Using that $e_1 = \pder{}{u}$, $e_2 = \pder{}{x} (\cosh u - h(x) \sinh u)^{-1}$, and $e_3 = \pder{}{v}$, we get that
    \begin{align}
    \label{eqn:g_f}
    g_f(X_1,X_2) &= -2 f(x) (\cosh u - h(x) \sinh u)^{-1} v \paren{\chevron{X_1, e_1} \chevron{X_2, e_2}
                + \chevron{X_1, e_2} \chevron{X_2, e_1}} \\
                &\quad + 2 f(x)(\cosh u - h(x) \sinh u)^{-1}  u \paren{\chevron{X_1, e_3} \chevron{X_2, e_2} + \chevron{X_1, e_2} \chevron{X_2, e_3}} \\
                &\quad + f(x)^2(\cosh u - h(x) \sinh u)^{-2}  (u^2 + v^2) \chevron{X_1, e_2} \chevron{X_2, e_2}
    \end{align}
    where $\chevron{\cdot, \cdot}$ is the inner product with respect to $g_{\mathbb{H}^2 \times \R}$.

    Fix any smooth vector fields $X_1, X_2$, and define $F := g_f(X_1, X_2)$, a function on $M$.
    By the above expression for $F$, we can observe that it has the following properties on $S_M$.
    \begin{enumerate}[(a)]
    \item $F$ is a rational function of functions of the following forms: $u$, $v$, $\cosh u$, $\sinh u$, $f^{(i)}(x)$, $h^{(i)}(x)$ (for $i = 0,1,2, \ldots$), or $\chevron{X, e_j}$ (for $j = 1,2,3$) where $X$ is a smooth vector field,
    \item the denominator of this rational function is bounded away from $0$ on any set where $\abs{u}$ is bounded,
    \item each term in the numerator of this rational function has a positive power of some $f^{(i)}(x)$.
    \end{enumerate}

    We can see that $F$ satisfies part $(b)$ since, by Proposition~\ref{prop:foliation_H_lipschitz}, $\abs{h(x)} \leq 1$ and hence 
    \[ \abs{\cosh u - h(x) \sinh u} \geq e^{\abs{u}}. \]
    Furthermore, the derivatives $e_{j_1} \cdots e_{j_k}(F)$ on $S_M$, also satisfies these three properties.
    This follows from directly computing the derivatives of functions of this form by using equations \eqref{eqn:start_cov_derivs}-\eqref{eqn:end_cov_derivs} and using that 
    \begin{align*}
        a(x,u,v) &= 0 \\ 
        \beta(x,u,v) &= (h(x) \cosh u - h(x) \sinh u)/(\cosh u - h(x) \sinh u).
    \end{align*}
    Hence all derivatives $Y_1 \cdots Y_k(F)$ for any smooth vector fields $Y_1, \ldots, Y_k$ on $S_M$ satisfies (a)-(c).

    We next claim that for any function $G$ that satisfies these three properties (a)-(c), $G$ extends continuously to all of $M$ with $G = 0$ on $M \setminus S_M$.
    Take a sequence of points $(x_k, u_k, v_k)$ in $S_M$ that converge to a point in $(x_*,u_*,v_*)$ in $M \setminus S_M$.
    By property (c) and our assumption on $f$, every term of the numerator must go to zero since all factors other than $h^{(i)}(x)$ are bounded as $k \rightarrow \infty$.
    By property (b), the denominator stays bounded away from 0 as $k \rightarrow \infty$.
    Hence $G(x_k, u_k, v_k) \rightarrow 0$ as $k \rightarrow \infty$.

    Since $F$ and all partial derivatives of $F$ on $S_M$ satisfy (a)-(c), $F$ extends smoothly to all of $M$ with $F = 0$ and $Y_1 \ldots Y_k(F) = 0$ on $M \setminus S_M$.
    Hence $g_f(X_1, X_2)$ is a smooth function for any fixed smooth vector fields $X_1, X_2$.
    Since $g_f(X_1, X_2)$ is bilinear in $X_1, X_2$, $g_f$ is a smooth tensor on $M$.

    Therefore $g = g_{\mathbb{H}^2 \times \R} + g_f$ is smooth and is of the form \eqref{eqn:g_curv_homog} on $\tilde S$.
    On $M \setminus \tilde S$, $g = g_{\mathbb{H}^2 \times \R}$ and hence $M$ has Ricci eigenvalues (-1,-1,0) everywhere.

\end{proof}

\begin{remark}
Since $h$ is bounded ($H$ is Lipschitz), we only need to check for the condition when all $\ell_i \geq 1$.
\end{remark}
\begin{remark}
If instead of a smooth metric, we wanted $g$ to be $C^K$, then the condition on $f$ and $h$ in equation \eqref{eqn:f_h_property} is needed only when $k + \sum_{i=1}^m \ell_i \leq K$.
In particular, for a $C^2$ metric, we need that $f, f', f'', fh', f(h')^2, fh'',$ and $f'h'$ go to zero at $x \not\in S$.
\end{remark}

\begin{remark}
By Proposition~\ref{prop:plane_bundle}, for any complete, simply connected $M$ with Ricci eigenvalues $(-1,-1,0)$ that is locally irreducible everywhere, there exists a $\gamma: \R \rightarrow M$ orthogonal to $\mathcal{F}$ and $f(x) := a(\gamma(x))$.
Let $H$ be the turning angle of $\gamma$.
This gives a candidate for a converse to Theorem~\ref{thm:irreducible_examples}.
However, it is not clear that such an $f$ and $h$ must satisfy the assumptions in equation~\eqref{eqn:f_h_property}.
\end{remark}

\section{Manifolds with Locally Reducible Points}
\label{sec:locally_reducible}

We now describe the structure of complete, simply connected manifolds $M$ which have Ricci eigenvalues $(-1, -1, 0)$ that may not be locally irreducible everywhere.
\begin{prop}
\label{prop:decomposition$V_i$}
Suppose that a complete, simply connected manifold $M$ has Ricci eigenvalues $(-1, -1, 0)$.
Then $M$ is decomposed as a union of disjoint regions $\{U_i\}$ such that each $U_i$ is either an open connected component of $M_{split}$ or a closed connected component of $M_{irred}$.
These satsify:
\begin{enumerate}[(A)]
    \item in the first case, we call $U_i$ a \emph{split region}, and $U_i$ is isometric to $\Sigma \times \R$ for $\Sigma \subset \mathbb{H}^2$ a connected subset of the hyperbolic plane whose boundary components are complete geodesics,
    \item in the second case, we call $U_i$ a \emph{non-split region}. In $U_i$, every point is locally irreducible and $C \not= 0$ on a dense, open subset. Furthermore, $U_i$ admits a Lipschitz foliation by the leafs of $\mathcal{F}$ and a  path $\gamma_i$ orthogonal to $\mathcal{F}$ which intersects every leaf exactly once.
\end{enumerate}
\end{prop}

\begin{proof}
    For each connected component of $M_{split}$ and each connected component of $M_{irred}$, we have a set $U_i$.
    Since $M \setminus M_{irred}$ is $M_{split}$, $M$ is the union of these disjoint sets.
    If $U$ is a non-split region, then its structure is given by Proposition~\ref{prop:plane_bundle}.

    It remains to be shown that a split region $U$ is isometric to $\Sigma \times \R$ for some simply connected $\Sigma \subset \mathbb{H}^2$.
    By the de Rham-type splitting result of \cite{graph_manifolds, twisted_products}, we know that $U$ is isometrically the product of $\Sigma \times \R$ for some surface $\Sigma$ with Gaussian curvature $-1$.
    Each boundary component of $U$ is also a boundary component of a non-split region.
    Since non-split regions have complete, flat, totally geodesic boundary components, so too must $U$.
    Since $M$ is simply connected, we have that $U$ is simply connected, since otherwise there would be a non-trivial covering of $M$ (obtained by gluing copies of $M \setminus U$ to the non-trivial cover of $U$).
    Hence $\Sigma$ is simply connected.
    To see that $\Sigma \subset \mathbb{H}^2$, we can consider its double $\Sigma \cup \Sigma$ glued along the geodesic boundary components.
    This is a complete surface with $K = -1$ and hence its universal cover is $\mathbb{H}^2$.
    Since $\Sigma$ is simply connected, its inclusion into the double then lifts to an inclusion in $\mathbb{H}^2$, as desired.
\end{proof}

\begin{figure}
\centering
\includegraphics[scale=0.8]{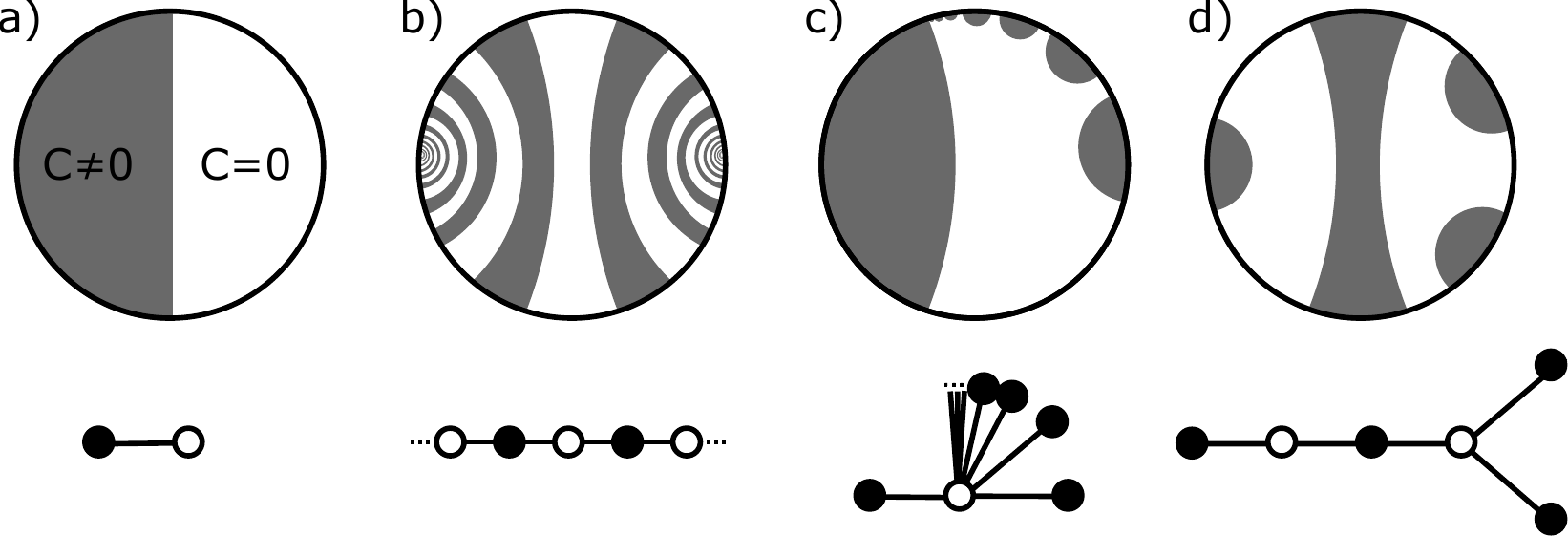}
\caption{Schematic examples of manifolds with Ricci eigenvalues $(-1, -1, 0)$ built off of trees.}
\label{fig:curv_homog_examples}
\end{figure}

\begin{example}
Figure~\ref{fig:curv_homog_examples} shows four possibilities for $M$, modelled after trees.
Each $M$ is drawn schematically in the Poincar\'e disk model of $\mathbb{H}^2$ where the split regions are white and the non-split regions are shaded.
Each non-split region may have any number of boundary components, including infinitely many.
Note that each split region has at most two (and possibly only one) boundary component.
We can construct these examples by taking non-split regions of the form \eqref{eqn:g_curv_homog} with $f(x) = 0$ outside of some interval.
Then these metrics are split outside of a strip and hence can be glued along their split regions.
\end{example}

\section{Topology}
\label{sec:topology}

In this section, we consider $M$ with Ricci eigenvalues $(-1, -1, 0)$ that may not be simply connected.
Since $\Sec \leq 0$, the universal cover $\widetilde M$ is always diffeomorphic to $\R^3$ and the topology of $M$ is determined by the fundamental group alone.

Our main result in this section will be Theorem~\ref{thm:pi1}, which states that any manifold with Ricci eigenvalues $(-1, -1, 0)$ and finitely generated fundamental group, has fundamental group that is a free group (unless $\widetilde M$ is split, i.e. isometric to $\mathbb{H}^2 \times \R$).

\begin{lemma}
    \label{lemma:def_A}
    Define $A(p,q)$ for $p,q \in \widetilde M$ by
    \[A(p,q) = \int \norm{C(\gamma')} dt = \int \abs{a(\gamma(t)) \chevron{\gamma', e_2}} dt \]
    integrating over the unique geodesic segment $\gamma$ from $p$ to $q$.
    Then $A(p,q)$ depends only upon the leaves $\mathcal{F}_p, \mathcal{F}_q$ and not on the choice of points on those leaves.
    Moreover, $A$ is an isometry invariant.
\end{lemma}
\begin{proof}
    First observe that the value inside the integral is well-defined since $e_2$ is defined except where $a = 0$, and the integral exists because $\abs{\chevron{\gamma', e_2}}$ is bounded by $\norm{\gamma'}$.
    Moreover, the definition is invariant of the choice of parametrization of $\gamma$.
    Since $\abs{a}$ is an isometry invariant, $A(p,g) = A(gp, gq)$ for any isometry $g$.

    Take points $p_2 \in \mathcal{F}_p$ and $q_2 \in \mathcal{F}_q$ and let $\gamma_2$ be the geodesic between the two points.
    Observe that every leaf of $\mathcal{F}$ that intersects $\gamma$ must also intersect $\gamma_2$, since each such leaf must have $\mathcal{F}_{p}$ and $\mathcal{F}_q$ on opposite sides of it.
    So the intervals where $a \not= 0$ on $\gamma$ are in bijection $\gamma$ with those on $\gamma_2$.
    To compute that integral in $A(p,q)$ we can restrict to the sum of the integrals over each interval where $\abs{a(\gamma(t))}$ is positive, and these regions are in bijection between $\gamma$ and $\gamma_2$.
    
    To show that $A(p,q) = A(p_2, q_2)$ we now need to show that the integral over these corresponding intervals of $\gamma$, $\gamma_2$ are equal.
    On a connected component with $a \not= 0$, there are coordinates $(x,u,v) \in (x_0, x_1) \times \R^2$ such that
    \begin{equation}
        \label{eqn:coordinates_a_nonzero}
        g = a^{-2}\; dx^2 + (du - v\; dx)^2 + (dv + u \; dx)^2
    \end{equation}
    where $a = C(x) (\cosh u - h(x) \sinh u)$ for some $C(x) \not= 0$, $\abs{h(x)} \leq 1$.
    In these coordinates, $T = \pder{}{v}, e_1 = \pder{}{u}$, the vector $e_2 = \abs{a} \paren{\pder{}{x} + v \pder{}{u} - u \pder{}{v}}$, and the leaves of $\mathcal{F}$ are the sets where $x$ is constant.

    We claim that for paths in this $a \not= 0$ region, $\int \abs{a(\gamma(t)) \chevron{\gamma', e_2}} dt$ is independent of the path taken between its starting and ending leaves, so long as it is increasing in $x$.
    Assume that the domain of $\gamma$, restricted to one such region, is $[0,1]$, then
    \begin{align*}
    \int_0^1 \abs{a \chevron{\gamma'(t), e_2}} dt
        &= \int_0^1 \abs{a^2 \chevron{\gamma'(t), \pder{}{x} + v \pder{}{u} - u \pder{}{v}}} dt \\
        &= \int_0^1 \abs{a^2 a^{-2} dx(\gamma'(t))} dt \\
        &= x(\gamma(1)) - x(\gamma(0))
    \end{align*}
    where the second equality follows by applying~\eqref{eqn:coordinates_a_nonzero}.
    Since the $x$ coordinate is constant on the leaves of $\mathcal{F}$, this integral depends only upon the leaves of the end points and hence $A(p,q) = A(p_2, q_2)$.
\end{proof}

\begin{remark}
    \label{remark:A}
    We think of $A$ as measuring a distance between any two leaves of $\mathcal{F}$.
    This distance measures the amount of rotation of the $T$ vector field between the two leaves.
    However, $A$ is only a pseudometric on $\mathcal{F}$.

    If $a$ is nonzero everywhere, then $A$ gives essentially the $x$-coordinate in the coordinates of the form in \eqref{eqn:coordinates_a_nonzero}.
    This allows $A$ to act as an extension of the $x$ coordinate to any non-split region, even those with $a = 0$ at some leaves.

    Moreover, the path integral in $A$ does not need to be a geodesic and a similar strategy shows that any path between $p,q$ will give the same value for $\int \abs{a \chevron{\gamma', e_2}} dt$, so long as there is no back-tracking, i.e. it never intersects the same leaf twice.
\end{remark}

We now work towards the proof of Theorem~\ref{thm:pi1} with lemmas that restrict the isometries that stabilize leaves of $\mathcal{F}$, non-split regions, and split regions of $M$, as well as a lemma for the case where $M$ is locally irreducible everywhere.

Let $\widetilde M$ be the universal cover of $M$.
Recall that since $\sec \leq 0$, if $G$ acts on $\widetilde M$ fixed point freely, then $G$ cannot have torsion.
If $g \in G$ has an invariant plane, i.e. $g(L) \subset L$ for some leaf $L$ in $\mathcal{F}$, then we will show that $g$ is trivial.
To do so, we may assume that $g$ acts by translations on the leaf.
Certainly $g$ acts by isometries on $L$ and has no fixed point and hence is either a translation or a glide reflection.
In the latter case, $g^2$ is a translation, so we may pass to $g^2$ instead of $g$, if necessary, since $g^2$ is trivial then so is $g$.
Similarly, if $g$ fixes a finite number of leaves, we will assume it acts by translations on all of them at  once.
We make use of these assumptions in the proof of the following lemmas.

Suppose that $G$ is any group of isometries acting fixed point freely on $\widetilde M$.
\begin{lemma}
    \label{lemma:two_fixed_planes}
    If $G$ fixes two distinct leaves $L_0, L_1 \in \mathcal{F}$, i.e. $G(L_i) \subset L_i$,
    then either $G$ is trivial or $L_0, L_1$ are boundary leaves of a split region.
\end{lemma}
\begin{proof}
    Assume there is some $g \not= e$ in $G$.
    First note that since the restriction $g|_{L_i}$ is an isometry of the flat leaf $L_i \simeq \R^2$, we may assume that $g$ acts by translation on each $L_i$, passing to $g^2$ if not.
    Pick a point $p_0 \in L_0$ and let $\gamma_0$ be the geodesic along which $g$ translates $L_0$.
    Recall $L_1$ is a totally geodesic plane, and hence is a convex subset of $\widetilde M$.
    Since $\sec \leq 0$, we have that $d(\cdot, L_1)$ is a convex function along any geodesic and in particular along $\gamma_0$.

    Moreover, $d(g^k(p_0), L_1) = d(p_0, L_1)$ and hence $d(\cdot, L_1)$ is constant along $\gamma_0$.
    Let $p_1$ be the unique point on $L_1$ closest to $p_0$ and $\gamma_1$ the geodesic in $L_1$ along which $g$ translates the point $p_1$.
    Then $g^k(p_1)$ is the unique point on $L_1$ closest to $g^k(p_0)$ and then $\gamma_0$ and $\gamma_1$ are parallel in the sense of having bounded (in fact constant) distance.
    Since $\sec \leq 0$, the union of all geodesics parallel to any geodesic is a convex subset isometric to $N \times \R$ for some closed convex subset $N$ of $M$,  see Lemma 2.4 in \cite{ballman}.
    So the geodesics $\gamma_0$ and $\gamma_1$ bound a flat strip, i.e. a totally geodesic submanifold isometric to $[0,\ell] \times \R$.

    Since this strip is flat, it must contain the nullity geodesics through each point of the strip.
    Since the nullity geodesics are complete, they are parallel in the strip, so $C(X) = 0$ for any vector $X$ in the strip.
    Note that the strip must be transverse to $e_1$ since if $e_1$ was in its tangent plane at one point, the strip would be contained in a single leaf of $\mathcal{F}$.
    Hence $C = 0$ on the strip.
    Since this holds for any $p_0 \in L_0$, we then have $C = 0$ on any point on a geodesic from $L_0$ to $L_1$, and so $L_0$ and $L_1$ must bound a split region.
\end{proof}

\begin{lemma}
    \label{lemma:no_fixed_nonsplit_regions}
    Suppose $V$ is any subset of $\widetilde M_{irred}$, with $V$ a strict subset of $\widetilde M$.
    If $G(V) \subset V$, then $G$ is trivial.
\end{lemma}
\begin{proof}
    Note that $G$ fixing $V$ implies that $G$ also fixes the entire non-split region (i.e., connected component of $M_{irred}$) containing $V$, so we may assume that $V$ is a non-split region rather than a subset of one.

    First, we consider the case where $V$ has two distinct boundary components, $L_0$ and $L_1$ which are leaves of $\mathcal{F}$.
    Pick $g \in G$.
    Either $g$ stabilizes each boundary leaves or $g^2$ does.
    Then Lemma~\ref{lemma:two_fixed_planes} shows that $g$ is trivial.

    Instead, suppose $V$ has exactly one boundary component $L_0$.
    (We allow the possibly that $V$ has no interior, so $V = L_0$.)
    Then $G(L_0) \subset L_0$.
    Since $V$ is a non-split region and so is in $M_{irred}$, we must have a sequence of points with $a \not= 0$ converging to $L_0$.
    In particular, infinitely many of those points lie on one side of $L_0$ and we will consider the leaves of $\mathcal{F}$ near $L_0$ on that side.
    Using $A$ defined in Lemma~\ref{lemma:def_A}, we consider $A(L_0, \mathcal{F}_p)$.
    Since $A(\cdot,\cdot)$ and $L_0$ are invariant under $G$, $A(L_0, \cdot)$ must be as well.

    Take $g \in G$ and $p_0 \in L_0$.
    Then consider the geodesics $\gamma, \mu$ starting at $p_0$ and $gp_0$, respectively, and orthogonal to $L_0$.
    There are infinitely many leaves of $\mathcal{F}$ intersecting $\gamma$, which limit to $p_0$.
    Since these leaves do not intersect $L_0$, they must limit towards being parallel to $L_0$.
    Therefore, there is an $\epsilon > 0$ so that all leaves within $\epsilon$ of $p_0$ along $\gamma$ also intersect $\mu$ near $gp_0$.
    The same is true for $\mu$, so that there is an $\epsilon > 0$ where the leaves along $\gamma$ all intersect $\mu$ and the leaves along $\mu$ all intersect $\gamma$.
    Taking $A(L_0, \cdot)$ along $\gamma$ and $\mu$ we then have that, for $s,t < \epsilon$, if $A(L_0, \mu(t)) = A(L_0, \gamma(s))$ then the leaves through $\mu(t)$ and $\gamma(s)$ must be the same.
    Since $A$ is an isometry invariant, $g$ must map such leaves to themselves.
    So $g$ fixes two leaves, and by \ref{lemma:two_fixed_planes}, $g$ is trivial.
\end{proof}

\begin{lemma}
    \label{lemma:split_region_free}
    Suppose that $U$ is a split region of $\widetilde{M}$ with at least one boundary component.
    If $G(U) \subset U$, then $G$ is a free group.
\end{lemma}
\begin{proof}
    Since $C = 0$ on $U$, $U$ is isometric to $\Sigma \times \R$ with $\Sigma$ a subset of $\mathbb{H}^2$ with complete geodesics for its boundary components.
    Suppose for contradiction that there is a non-trivial $g \in G$ such that $g$ fixes a point $p \in  \Sigma$.
    Then let $r = \inf_{\gamma_j} d(p, \gamma_j)$ where $\curly{\gamma_j}$ is the set of boundary components of $\Sigma$.
    There must be at least one boundary component that realizes the infimum and, moreover, only finitely many do, since the boundary components are complete geodesics in $\mathbb{H}^2$.

    Then $g$ must act on the set $\curly{\gamma_j | d(p, \gamma_j) = r}$ of those boundary components.
    Since the set is finite, there is some $k > 0$ such that some $\gamma_j$ is invariant under $g^k$.
    But then the  boundary $\gamma_j \times \R$ of $U$ is invariant under $g^k$.
    By the previous lemma, we know the only such isometries are trivial.
    Then $g$ has order at most $k$, but $G$ cannot have torsion.
    So $G$ must act fixed-point freely on $\Sigma$

    Similarly, we can see that $G$ acts properly discontinuously on $\Sigma$.
    Suppose that $p_0 \in U$ and there is a sequence of distinct points $p_i = g_i(p_0)$, $g_i \in G$ with $g_i \not= g_j$ with $p_i \rightarrow p_* \in U$.
    Then let $\gamma_*$ be any geodesic of minimal distance to $p_*$ and let $D = d(\gamma_*, p_*)$.
    For $\epsilon > 0$, choose $N$ so that $d(p_i, p_*) < \epsilon$ for $i > N$.
    Each $g_i$ has $g_i^{-1}(\gamma_*)$ a boundary geodesic and in particular since $d(p_i, p_*) < \epsilon$, $d(g_i^{-1}(\gamma_*), p_0) < D + \epsilon$.
    The set of boundary geodesics that are distance at most $D + \epsilon$ from $p_0$ is finite.
    Hence there exists $j > k > N$ such that there is a geodesic $\gamma_0$ of distance at most $D + \epsilon$ from $p_0$ so that $g_j(\gamma_0)$ and $g_k(\gamma_0)$ are both $\gamma_*$.
    Then $g_j^{-1} g_k$ must fix $\gamma_0$.
    This is a contradiction with the previous lemma.
    So $G$ must act properly discontinuously as well as fixed-point freely.

    Now $\Sigma$ is an open surface that is contractible (since it is a convex subset of $\mathbb{H}^2$) and $G$ acts on it fixed-point freely and properly discontinuously.
    So $G$ is the fundamental group of $\Sigma / G$, a non-compact surface.
    Hence $G$ is a free group, by a well-known fact that the fundamental group of any non-compact surface is free.
    See Section 4.2.2 of \cite{stillwell} for reference.
\end{proof}

\begin{lemma}
    \label{lemma:pi1_irreducible}
    Suppose that a complete manifold $M$ has constant Ricci eigenvalues $(-1, -1, 0)$ and is everywhere locally irreducible.
    Then $\pi_1(M)$ is either trivial or $\Z$.
    \end{lemma}
\begin{proof}
    We again use Lemma~\ref{lemma:def_A}.
    Specifically, pick a leaf $L_0$ of $\mathcal{F}$ of $\widetilde M$.
    Then define $A(L_p)$ by $\pm A(L_0, \mathcal{F}_p)$, choosing the positive sign on one side of $L_0$ and negative on the other.
    That $M$ is locally irreducible implies that $A$ is injective, i.e. no two distinct leaves map to the same value.

    Next, since $A$ does not depend upon the path used to compute it, $A(p,q)$ has the following property:
    for any $p_0, p_0', q_1, q_2 \in \widetilde M$ such that $\mathcal{F}_{p_0}$ has all of $p_0', q_1$, and $q_2$ on one side of it,
    \begin{align*}
        A(p_0, q_1) - A(p_0, q_2)
            &= (A(p_0, p_0') + A(p_0', q_1)) - (A(p_0, p_0') + A(p_0', q_1))\\
            &= A(p_0', q_1) - A(p_0', q_2).
    \end{align*}
    This equality then extends to all $p_0, p_0'$ by applying it twice $p_0, p_0''$ and then $p_0'', p_0'$ for some $p_0''$.

    This implies that $A(L_p)$ is equivariant under isometries:
    \begin{align*}
    A(g(L_p)) - A(g(L_q))
        &= A(L_0, g(L_p)) - A(L_0, g(L_q)) \\
        &= A(g(L_0), g(L_p)) - A(g(L_0), g(L_q)) \\
        &= A(L_0, L_p) - A(L_0, L_q) \\
        &= A(L_p) - A(L_q).
    \end{align*}
    Then $\pi_1(M)$ acts fixed-point freely on the image of $A$ since otherwise a non-trivial $g \in \pi_1(M)$ would fix $L_0$ which contradicts Lemma~\ref{lemma:no_fixed_nonsplit_regions}.
    Therefore if $\pi_1(M)$ is non-trivial, $A$ must be surjective on $\R$.
 
    Next, we want to show that $\pi_1(M)$ acts properly discontinously on $\R$ and hence is either trivial or $\Z$.
    Suppose not.
    Then the orbit of any $x \in \R$ under $\pi_1(M)$ is dense in $\R$.
    Hence, if some leaf $P$ has $a = 0$ on it, then $a = 0$ on all leaves, by continuity of $a$, and so $a = 0$ on all of $M$.
    This contradicts the assumption that $M$ is locally irreducible.
    So $a \not= 0$ on $M$.
    We may now assume without loss of generality that $a > 0$ on $M$.

    Hence we have a smooth foliation with coordinates $(x,u,v)$ as in Proposition~\ref{prop:g_curv_homog}.
    Fix some $p_0 \in \widetilde M$ with $p_0 \in L_0$ and call $p_0 = (0,0,0)$.
    By assumption, there are $g_k \in \pi_1(M)$ such that $g_k(p_0)$ are in leaves $L_k$ with $A(L_k) \rightarrow 0$ as $k \rightarrow \infty$.
    Since $\pi_1(M)$ acts properly discontinuously on $\widetilde M$, we must have that only finitely many of $p_k := g_k(p_0)$ are in any compact neighborhood of $p_0$.
    Let $q_k$ be the point on each leaf $L_k$ so that $q_k$ lies on the curve $(x,0,0)$, so then $q_k \rightarrow p_0$ as $k \rightarrow \infty$.
    Letting $(x_k, u_k, v_k) = p_k$, if $u_k \rightarrow \pm \infty$ as $k \rightarrow \infty$ (on any subsequence), then $g_k^{-1}(q_k)$ must have $u$-coordinate $-u_k$ and lies in $L_0$.
    But since we know the form of $a(x,u,v) = f(x)/(\cosh u - h(x) \sinh u)$, either $a \rightarrow 0$ or $a \rightarrow \infty$ for $u \rightarrow \infty$ and $x$ fixed.
    This is a contradiction since $a(0,-u_k,-v_k)$ must be the same as at $a(q_k)$ by the isometry $g_k$ and $a(q_k) \rightarrow a(p_0)$ which is non-zero and finite.

    Hence $u_k$ must be bounded.
    Then $v_k$ must diverge instead.
    As in \cite{graph_manifolds}, we note that
    \begin{equation*}
        T(e_2(a)) = e_2(T(a)) + [T,e_2](a) = (\nabla_T e_2 - \nabla_{e_2} T)(a) = a e_1(a)
    \end{equation*}
    and that $T(e_1(a)) = e_1(T(a)) + [T,e_1](a) = 0$.
    Hence $e_2(a) = a e_1(a) v + d$ for some $d$ with $d, a,$ and $e_1(a)$ all independent of $v$.
    Therefore $e_2(a)$ at $g_k^{-1}(q_k)$ must diverge as $k \rightarrow \infty$ since $u$ is bounded and $x=0$.
    But this means that $e_2(a)$ must diverge at $p_0$ since $q_k \rightarrow p_0$ and $e_2(a)$ is an isometry invariant up to sign.
    This gives a contradiction since $e_2(a)$ must be finite at any point.
    Hence $\pi_1(M)$ must actually act properly discontinuously on $\R$, and hence is trivial or $\Z$.
\end{proof}

Now we return to the Theorem~\ref{thm:pi1}.
Our strategy is to build an integer-valued Lyndon length function $N: \pi_1(M) \rightarrow \mathbb{N}$ \cite{lyndon} which can be thought of as an integer-valued version of $g \mapsto A(p, g(p))$.
Lemma~\ref{lemma:pi1_irreducible} covers the case where $M$ is locally irreducible everywhere, so we assume for the remainder of this section that $M$ is locally reducible at some points.
Moreover, we assume that $\widetilde M$ is irreducible, so $a \not= 0$ at some point.

\begin{definition}
    Fix a point $p_0 \in \widetilde M$ such that $a \not= 0$.
    Let $\mathcal{V}$ be a non-empty finite collection of connected components of $M_C$.
    We require that $\mathcal{V}$ must include the connected component that contains $p_0$.
    Let $\mathcal{L}_0$ be the set of boundary leaves of all $V \in \mathcal{V}$ and let $\mathcal{L} := \curly{g(L) |g \in \pi_1(M), \quad L \in \mathcal{L}_0}$ be the images of those leaves under the action of $\pi_1(M)$.
\end{definition}
\begin{definition}
    For a choice of $\mathcal{V}$, define $N: \pi_1(M) \rightarrow \N$ so that $N(g)$ is the number of leaves of $\mathcal{L}$ crossed by the geodesic from $p_0$ to $g(p_0)$.
\end{definition}
We will vary the choice of $\mathcal{V}$ and write $N_\mathcal{V}$ to clarify when necessary.
\begin{lemma}
    \label{lemma:def_N}
    $N(g)$ is finite.
\end{lemma}
\begin{proof}
    Suppose $g$ is non-trivial.
    For each $V \in \mathcal{V}$, let $\epsilon_V =  \int \abs{a \chevron{\gamma', e_2}} dt > 0$ along any geodesic $\gamma$ from one boundary component of $V$ to the other.
    Then define $\epsilon = \min_{V \in \mathcal{V}} \epsilon_V$ so $\epsilon > 0$.
    Then we claim that 
    \[ \epsilon (N(g) - 2) \leq A(p_0, g(p_0)), \]
    and hence $N$ is finite.
    Let $\gamma$ be the geodesic from $p_0$ to $g(p_0)$.
    Consider the open regions where $a \not= 0$ that contain $g(p_0)$ for some $g \in \pi_1(M)$.
    Then $\gamma$ intersects some number of these.
    If it intersects $g(V)$ for some $V \in \mathcal{V}$, it must cross from one boundary to the other, unless $g(V)$ contains either of $p_0$ or $g(p_0)$.
    So $\gamma$ crosses $(N(g) - 2)/2$ such regions, with each crossing contributing at least $\epsilon$ to $A(p, g(p_0))$.
    Hence the inequality holds and $N$ is finite.
\end{proof}

\begin{lemma}
    \label{lemma:Lyndon_length}
    Define the overlap function $s(g,h) = \frac 1 2 \brak{N(g) + N(h) - N(gh^{-1})}$.
    The function $N: \pi_1(M) \rightarrow \Z$ is a Lyndon length function (see \cite{lyndon}) in the sense that it satisfies the following:
    \begin{enumerate}[(I)]
    \item $N(g) = 0$ iff $g$ is trivial,
    \item $N(g^{-1}) = N(g)$,
    \item $s(g,h) \geq 0$,
    \item $s(g,h) < s(g,\ell)$ implies that $s(h, \ell) = s(g,h)$, and 
    \item $s(g,h) + s(g^{-1}, h^{-1}) > N(g) = N(h)$ implies that $g = h$.
    \end{enumerate}
\end{lemma}

\begin{proof}
    Let $G := \pi_1(M)$.
    For (I), note that if $N(g) = 0$, then the geodesic from $p_0$ and $g(p_0)$ must never reach a point with $a = 0$.
    Then the subgroup generated by $\chevron{g} \subset G$ must leave $V$ invariant, where $V$ is the non-split region containing $p_0$.
    So Lemma~\ref{lemma:no_fixed_nonsplit_regions} implies that $g$ is trivial.

    For (II), if $\gamma$ is the geodesic from $p_0$ to $g(p_0)$, then $g^{-1}(\gamma)$ is the geodesic from $g^{-1}(p_0)$ to $p_0$ and $\mathcal{L}$ is isometry invariant.

    For (III) and (IV), we give $s(g,h)$ a geometric interpretation.
    Observe that $s(g,h)$ counts the number of leaves of $\mathcal{L}$ that lie on both the geodesic from $p_0$ to $g(p_0)$ and the one from $p_0$ to $h(p_0)$.
    To see this, first observe that $N(gh^{-1})$ is equal to the number of leaves of $\mathcal{L}$ crossed by the geodesic from $h(p_0)$ to $g(p_0)$, since $h$ acts by an isometry.
    Assume that $g$ and $h$ are non-trivial and distinct.
    Consider all the leaves crossed by any of the three geodesics between $p_0, g(p_0)$ and $h(p_0)$.
    Each of these leaves divides $\widetilde M$ into two connected components, with two of $\{p_0, g(p_0), h(p_0)\}$ on one side and one of the other.
    None of the leaves can have all three points on one side, since it crosses the geodesic between two of the points.

    Any leaf with $p_0$ and $g(p_0)$ on the same side contributes $+1$ to the $N(g)$ term and $-1$ to the $N(gh^{-1})$, and hence contributes 0 to $s(g,h)$.
    Similarly, for leaves with $p_0$ and $h(p_0)$ on the same side.
    For leaves with $g(p_0)$ and $h(p_0)$ on the same side, these contribute $+1$ to $N(g)$ and $+1$ to $N(h)$, so account for $+1$ to $s(g,h)$.
    Then $s(g,h)$ is the number of leaves with $g(p_0)$ and $h(p_0)$ on one side and $p_0$ on the other, which equals the number of leaves that intersect both geodesics from $p_0$ to $g(p_0)$ and $h(p_0)$.
    Lastly note that if $g$ or $h$ are trivial, then $s(g,h) = 0$, and if $g = h$ then $s(g,h) = N(g) = N(h)$.
    Hence $s$ is non-negative and (III) holds.

    For (IV), let $\gamma_g$, $\gamma_h$, and $\gamma_\ell$ be the geodesics from $p_0$ to $g(p_0), h(p_0)$ and $\ell(p_0)$, respectively.
    See Figure~\ref{fig:Lyndon_length}.
    We consider the leaves intersecting $\gamma_g$ as ordered from closest to $p_0$ to furthest.
    Note that if one leaf $L \in \mathcal{L}$ intersects both $\gamma_g$ and $\gamma_h$, then all earlier leaves intersecting $\gamma_g$ must also intersect $\gamma_h$.
    This follows since $L$ has $g(p_0)$ and $h(p_0)$ on one side and the earlier leaves on the other, so all earlier leaves also have $g(p_0)$ and $h(p_0)$ on the same side.
    Hence the first $s(g,h)$ leaves on $\gamma_g$ must also intersect $\gamma_h$ and the first $s(g,\ell)$ intersect $\gamma_\ell$.
    Since $s(g,h) < s(g,\ell)$, all leaves intersecting both $\gamma_g$ and $\gamma_h$ also intersect $\gamma_\ell$ and so $s(g,h) \leq s(h, \ell)$.
    Since $s(g,h) \not= s(g,l)$, the $s(g,h)+\nth{1}$ leaf intersecting $\gamma_g$ must intersect $\gamma_\ell$ but not $\gamma_h$.
    Hence the $s(g,h)+\nth{1}$ leaf on $\gamma_\ell$ intersects $\gamma_g$ but not $\gamma_h$ and so $s(h,\ell) < s(g,h) + 1$ which gives (IV).

    \begin{figure}
    \centering
    \includegraphics[scale=0.32]{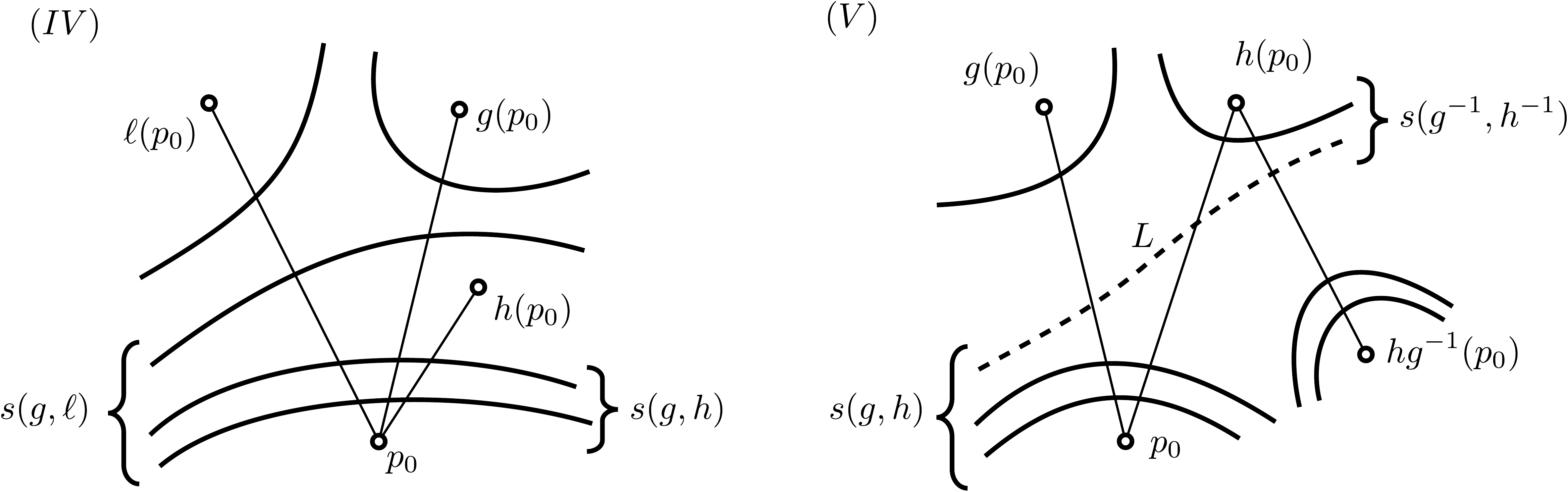}
    \caption{Left, diagram of property (IV) showing possible leaves of $\mathcal{L}$ (thick lines) and the geodesics from the point $p_0$. Right, diagram of property (V) showing the leaf $L$ (dashed line) intersecting all three geodesics.}
    \label{fig:Lyndon_length}
    \end{figure}

    For property (V), assume that $g,h \in \pi_1(M)$ satisfy $s(g, h) + s(g^{-1}, h^{-1}) > N(g) = N(h)$.
    Note that $s(g,h)$ counts the number of leaves of $\mathcal{L}$ that intersect both the geodesics $\gamma_g$ and $\gamma_h$ from $p_0$ to $g(p_0)$ and $h(p_0)$.
    By applying the isometry $h$ first, we see that $s(g^{-1}, h^{-1})$ counts the number of leaves intersecting both the geodesic from $h(p_0)$ to $p_0$ and the geodesic from $h(p_0)$ to $h g^{-1}(p_0)$.
    So these both count leaves intersecting the geodesic $\gamma_h$.
    The total number of leaves of $\mathcal{L}$ interescting $\gamma_h$ is $N(h) < s(g,h) + s(g^{-1}, h^{-1})$.
    Let $\gamma_{hg^{-1},h}$ be the geodesic from $hg^{-1}(p_0)$ to $h(p_0)$.
    By the pigeonhole principle, at least one leaf $L$ intersects all three of $\gamma_g, \gamma_h$ and $\gamma_{h,hg^{-1}}$.
    In particular, $L$ can be taken to be the $s(g,h)$th leaf along the $\gamma_g$ and $\gamma_h$.
    Moreover, it is the $N(h) - s(g,h)+1$st leaf going backwards along $\gamma_h$ and therefore also the $N(h) - s(g,h)+1$st leaf going backward along $\gamma_{hg^{-1}, h}$.
    Since $N(g) = N(h)$, it is also the $N(h) - s(g,h)+1$st leaf going backwards along $\gamma_g$.
    Since the isometry $h g^{-1}$ takes $\gamma_g$ to $\gamma_{hg^{-1}, h}$, it takes $L$ to $L$.
    Then Lemma~\ref{lemma:no_fixed_nonsplit_regions} implies that $g^{-1}h$ is trivial.
\end{proof}

If $g \in \pi_1(M)$ is such that $N(g^2) \leq N(g)$, then we call $g$ \emph{non-Archimedean}.
If $g$ is non-Archimedean, then we define 
\[ \mathcal{N}_x := \curly{y \in \pi_1(M): N(x y^{-1}) \leq N(x) = N(y)} \cup \{e\} \]
where $e$ is the identity element.
By~\cite{lyndon}, $\mathcal{N}_x$ satisfies the following properties:
\begin{enumerate}
    \item $\mathcal{N}_x$ is a group all of whose elements are non-Archimedean,
    \item either $\mathcal{N}_x = \mathcal{N}_y$ or $\mathcal{N}_x \cup \mathcal{N}_y = \{e\}$, and
    \item $y \in \mathcal{N}_x$ and $x \not= y$ implies that $s(x,y) = \frac 1 2 N(x) = \frac 1 2 N(y)$
\end{enumerate}

The main result, Theorem 7.1, of ~\cite{lyndon}, states that for any group $G$ with a Lyndon length function, there exists a decomposition of $G$ as the free product of subgroups that are  either
\begin{enumerate}
    \item $\mathcal{N}_x$ for a non-Archimedian $x \in G$, or
    \item an infinite cyclic subgroup generated by an Archimedean $x \in G$.
\end{enumerate}

Considering $N$ as a discrete approximation of $g \mapsto A(p_0, gp_0)$, the next lemma classifies analogs of the subgroups $\mathcal{N}_x$ using $A$ instead of $N$.
We will then refine $N$ repeatedly to approximate $A$ sufficiently.
\begin{lemma}
    \label{lemma:nonarchimedean_free}
    Suppose that $G \subset \pi_1(M)$ is such that for every non-trivial $g,h  \in G$, 
    \[ A(p_0, gp_0) = A(p_0, hp_0). \]
    Then $G$ is a free group.
\end{lemma}

\begin{proof}
    Note that for any $g,h \in G$, $A(p_0, gp_0) = A(p_0, hp_0)$ implies that the leaf $\mathcal{F}_{gp_0}$ must have $p_0$ and $hp_0$ on the same side.
    Thefore there must be a split region $U \subset \widetilde M_{split}$ such that $\widetilde M \setminus U$ has $p_0, gp_0,$ and $h p_0$ on three separate connected components.
    To see this, take $U$ to be the split region with the following three leaves of $\mathcal{F}$ in its boundary:
    \begin{itemize}
        \item $L_{p_0}$, the last leaf along the geodesic from $p_0$ to $gp_0$ which has both $gp_0$ and $hp_0$ on the same side,
                (which is therefore also the last leaf from $p_0$ to $hp_0$ with this property),
        \item $L_{gp_0}$, the last leaf from $gp_0$ to $p_0$ which has both $p_0$ and $hp_0$ on one side (which is therefore also the last leaf from $gp_0$ to $h p_0$ with this property), and
        \item $L_{hp_0}$, the last leaf from $hp_0$ to $p_0$ which has both $p_0$ and $gp_0$ on one side (which is therefore also the last leaf from $hp_0$ to $gp_0$ with this property).
    \end{itemize}
    Then each of $p_0$, $gp_0$ or $hp_0$ is separated from $U$ by the corresponding boundary leaf.
    Moreover, this is the unique $U$ to have this property, since any other split region is contained entirely in one component of $\widetilde M \setminus U$ and therefore has two of $p_0, gp_0,$ and $hp_0$ on one side.

    Let $x = A(p_0, L_{p_0})$, $y = A(gp_0, L_{gp_0})$, and $z = A(hp_0, L_{hp_0})$.
    By construction of the leaves, note that $A(p_0, gp_0) = x+y$, $A(p_0, hp_0) = x + z$ and $A(gp_0, hp_0) = y+z$.
    By the assumption on $G$, and the fact that $A$ is invariant under $g$, we have $x+y = x+z = y+z$.
    Therefore $x=y=z = \frac 1 2 A(p_0,gp_0)$.
    So $L_{p_0}$ is independent of $g$ and $h$, and therefore $U$ is independent of $g$ and $h$, depending only upon $p_0$.
    Replacing $p_0$ with $gp_0$, we see that $U$ is determined only by the orbit $Gp_0$.

    Therefore the action of $G$ on $\widetilde M$ must stabilize $U$.
    By Lemma~\ref{lemma:split_region_free}, $G$ must then be a free group.
\end{proof}

\begin{figure}
    \includegraphics[scale=0.3]{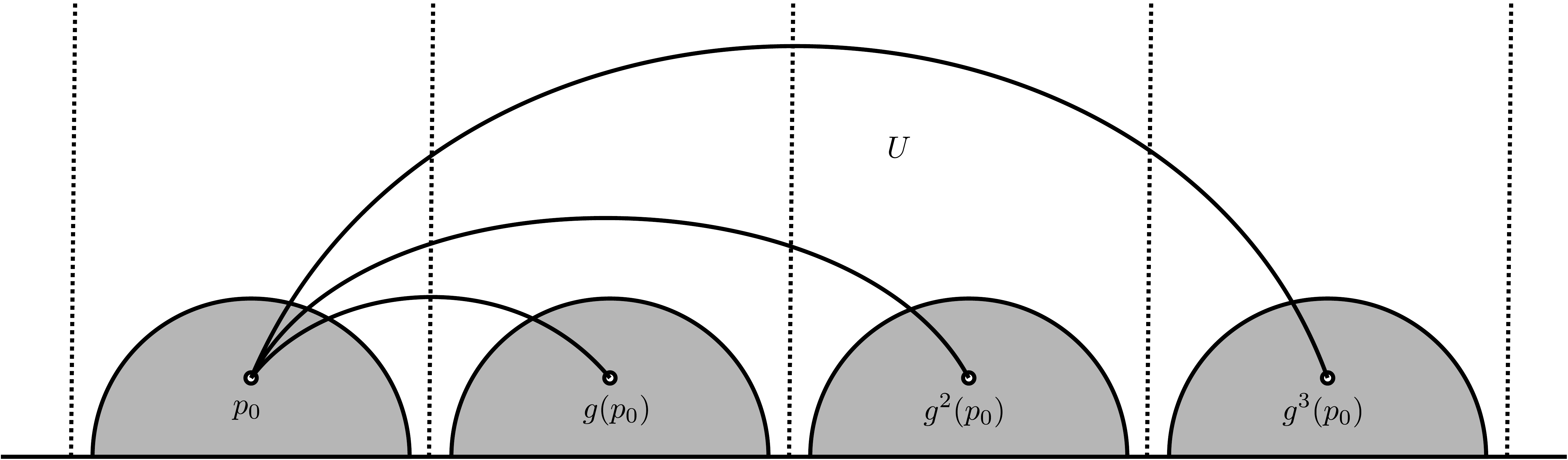}
    \caption{Examples of non-Archimedean isometries.  Consider the upper half-plane model of $\mathbb{H}^2 \times \R$, modified to have non-split regions in gray.  The isometry $g$ acts by translation.  Using $\mathcal{V}$ containing just the shaded gray regions, we can see that $N_{\mathcal{V}}(g) = N_{\mathcal{V}}(g^2) = 2$ by tracing the geodesics from $p_0$ (solid lines).  In this case the split region $U$ is fixed by $g$.  If instead there were also non-split regions added at the dashed lines, then there would be no such fixed region, but $A(p_0, gp_0)$ no longer equals $A(p_0, g^2p_0)$.  In this case, we then include these non-split regions in $\mathcal{V}$ so that $g$ is no longer non-Archimedean.}
    \label{fig:nonArchimedean}
\end{figure}

\begin{proof}[Proof of Theorem~\ref{thm:pi1}]
The strategy is to apply the result of \cite{lyndon}, Theorem 7.1, described above, to $N_{\mathcal{V}}$, for an appropriate choice of $\mathcal{V}$.
We start with the smallest possible $\mathcal{V}$ and then add to it to further refine the free decomposition of $\pi_1(M)$, until all subgroups in the free decomposition are themselves free.
As motivation of the following strategy, see Figure~\ref{fig:nonArchimedean}, which gives examples of non-Archimedean isometries.

Let  $\mathcal{V}$ be $\curly{gV: g \in \pi_1(M)}$ where V is the component of $M_C$ that contains $p_0$.
Applying the theorem of Lyndon gives a decomposition of $G = \pi_1(M)$ as a free product of groups $G_1, \ldots, G_k$.
If $G_i$ is a subgroup which is the cyclic group generated by an Archimedean element, then $G_i$ is free since $\pi_1(M)$ is torsion-free.
If not, then $G_i$ is $\mathcal{N}_x$ for some non-Archimedean $x \in G$.
If that $G_i$ has $A(p_0, gp_0) = A(p_0, hp_0)$ for all $g,h \in G_i$, then by Lemma~\ref{lemma:nonarchimedean_free}, it too is free.

If not, then take $g, h \in G_i$ with $A(p_0, gp_0) \not= A(p_0, hp_0)$.
For each connected component $W$ of $M_C$, let $m$ be the number of its images under $\pi_1(M)$ which intersect the geodesic from $p_0$ to $gp_0$ and $n$ be the number that intersect the geodesic from $p_0$ to $hp_0$.
Since $A(p_0, gp_0) \not= A(p_0, hp_0)$, there must be at least one $W$ with $n \not= m$.
We then take $\mathcal{V}' = \mathcal{V} \cup \curly{gW: g \in \pi_1(M)}$.
Then $N_{\mathcal{V}'}(g) \not= N_{\mathcal{V}'}(h)$.
Therefore $G_i$ is not a non-Archimedean subgroup for $N_{\mathcal{V}'}$.

We can then apply the result of \cite{lyndon} again to $G_i$ using $N_{\mathcal{V}'}$ to get that either $G_i$ is free or a free product of at least two non-trivial subgroups.
We can proceed recursively on these subgroups if necessary.
This recursion must stop after finitely many steps, since $\pi_1(M)$ is finitely generated and therefore is the free product of at most finitely may subgroups, by the Grushko theorem.
Therefore $\pi_1(M)$ is a free product of free groups and so itself is free.

Conversely, any countable free group can be achieved as the fundamental group of some $M$ with Ricci eigenvalues $(-1,-1,0)$.
First, recall that every countably generated free group is a subgroup of $F_2$, the free group with two generators.
Then it suffices to show that $\pi_1(M) = F_2$ is possible.
We can construct such a manifold by taking a subset $U$ of $\mathbb{H}^2 \times \R$ with four boundary components $P_1, \ldots, P_4$ that are totally geodesic planes.
Then take any two non-split regions $V_1, V_2$ with two boundary planes each and $a \rightarrow 0$ to infinite order on these boundary planes (and $h = 0$) constructed by Theorem~\ref{thm:irreducible_examples}.
Glue the two boundaries of $V_1$ to $P_1$ and $P_2$ and the two boundaries of $V_2$ to $P_3$ and $P_4$.
Then $M$ deformation retracts onto a wedge of two circles and $\pi_1(M) = F_2$.
\end{proof}

\begin{example}
There is a $\Z$ action on any metric of the form \eqref{eqn:g_curv_homog} if $f,h$ are periodic of the same period.
Then the $\Z$ action is just by translation in $x$ by the period of $f$ and $h$, and $M$ is locally irreducible everywhere if $f$ is never zero in a neighborhood.
\end{example}

\begin{example}
Note that the assumption in Theorem~\ref{thm:pi1} that $\widetilde M$ is irreducible is necessary.
For example, $\Z \times \Z$ acts on the product metric $\mathbb{H}^2 \times \R$ with one $\Z$ acting on each factor, or a surface group can act on the $\mathbb{H}^2$ factor.
\end{example}

\printbibliography

\end{document}

%% file: base_commands.tex
\newcommand{\R}[0]{\mathbb{R}}

\newcommand{\Z}[0]{\mathbb{Z}}
\newcommand{\N}[0]{\mathbb{N}}

\newcommand{\abs}[1]{\left \vert #1 \right \vert}
\newcommand{\brak}[1]{\left [ #1 \right ]}
\newcommand{\paren}[1]{\left ( #1 \right )}
\newcommand{\curly}[1]{\left \{ #1 \right \}}
\newcommand{\norm}[1]{\left \| #1 \right \|}
\newcommand{\chevron}[1]{\left \langle #1 \right \rangle}

\newcommand{\pder}[2]{\tfrac{\partial#1}{\partial#2}}

\DeclareMathOperator{\im}{Im}

\DeclareMathOperator{\Sec}{sec}
\DeclareMathOperator{\Scal}{Scal}

\DeclareMathOperator{\trace}{tr}

\DeclareMathOperator{\grad}{grad}

\newtheorem{theorem}[equation]{Theorem}
\newtheorem*{theorem*}{Theorem}
\newtheorem{prop}[equation]{Proposition}
\newtheorem*{prop*}{Proposition}
\newtheorem{lemma}[equation]{Lemma}
\newtheorem*{lemma*}{Lemma}
\newtheorem{corollary}[equation]{Corollary}
\newtheorem*{corollary*}{Corollary}

\newtheorem*{question*}{Question}

\newtheorem*{problem*}{Problem}

\newtheorem*{conjecture*}{Conjecture}

\theoremstyle{remark}
\newtheorem{remark}[equation]{Remark}
\theoremstyle{definition}
\newtheorem{definition}[equation]{Definition}
\newtheorem{example}[equation]{Example}
